\documentclass[a4paper,11pt,reqno]{amsart}
\usepackage{enumitem}
\usepackage{amssymb,amsmath,amsfonts,amsthm}
\usepackage{mathrsfs}
\usepackage{amscd}
\usepackage[active]{srcltx}
\usepackage{verbatim}
\usepackage[colorlinks,linkcolor={blue},citecolor={blue},urlcolor={black}]{hyperref}
\usepackage[utf8]{inputenc}
\usepackage{xspace}
\usepackage{tikz}
\usetikzlibrary{backgrounds}
\usetikzlibrary{decorations.pathmorphing}
\usetikzlibrary{arrows}
\usetikzlibrary{shapes.geometric}
\usetikzlibrary{calc}
\usetikzlibrary{arrows.meta,bending,positioning}
\usepackage{pgfplots}
\pgfplotsset{compat=1.15}
\usepackage{fancyhdr}

\theoremstyle{plain}
\newtheorem{theorem}{Theorem}
\newtheorem{lemma}{Lemma}[section]

\newtheorem{proposition}[lemma]{Proposition}

\newtheorem{definition}[lemma]{Definition}

\newtheorem*{definition*}{Definition}

\theoremstyle{remark}
\newtheorem{remark}[lemma]{Remark}

\newtheorem*{claim*}{Claim}
\newtheorem*{remark*}{Remark}
\newtheorem*{example*}{Example}

\newtheorem*{notation*}{Notation}

\numberwithin{equation}{section}

\def\Z{{\mathbb Z}}
\def\R{{\mathbb R}}

\newcommand{\RR}{\mathbb R}

\newcommand{\RRd}{\mathbb{R}^d}
\newcommand{\TTd}{\mathbb T^d}
\newcommand{\TND}{\mathbb T_N^d}

\newcommand{\E}{{\mathbb E}}
\renewcommand{\P}{{\mathbb P}}
\newcommand{\PP}{\mathbb{P}}

\newcommand{\Ent}{{\rm Ent}}

\renewcommand{\phi}{\varphi}
\newcommand{\eps}{\varepsilon}
\newcommand{\dd}{\mathrm d}
\newcommand{\cI}{\mathcal I}
\newcommand{\cJ}{\mathcal J}
\newcommand{\cR}{\mathcal R}
\newcommand{\cL}{\mathcal L}

\newcommand{\cM}{\mathcal M}
\newcommand{\cX}{\mathcal X}

\newcommand{\cU}{\mathcal U}

\newcommand{\QQ}{\mathbb{Q}}
\newcommand{\EE}{\mathbb{E}}
\newcommand{\TT}{\mathbb{T}}
\newcommand{\ip}[1]{\left\langle #1\right\rangle}
\providecommand{\abs}[1]{\left\lvert#1\right\rvert}
\providecommand{\norm}[1]{\left\Vert#1\right\Vert}

\newcommand{\step}[1]{{\textbf{Step #1}}}

\makeatletter

\title{Pathological Large Deviations of the KMP Process in Dimension $d\ge 2$}
\author{Daniel Heydecker}
\address{}

\makeatother
\begin{document}

\begin{abstract}
We study dynamic large deviations for the Kipnis--Marchioro--Presutti process on the discrete torus $\TND$. By recasting the candidate rate function in terms of a linear hyperbolic-parabolic equation with rough drift, we show that pathological trajectories appear in the large deviations with finite rate in any dimension $d\ge 2$. Along the way, we rigorously validate the lower bound derived by Bertini-Gabrielli-Lebowitz in any dimension.
\end{abstract}

\maketitle

\tableofcontents

\section{Introduction \& Main Results}

The Kipnis--Marchioro--Presutti (KMP) process is a simple and canonical Markov process $\xi^N_t$ describing the conduction of heat through a solid \cite{kipnis1982heat}. This paper is dedicated to the problem of understanding the large deviations of $\xi^N_t$ away from \begin{equation}
	\label{eq:heat} \partial_tu=\frac12\Delta u; \qquad x\in \TTd, t\in [0,T],
\end{equation} which is an important step both in justifying the macroscopic fluctuation theory \cite{bertini2006non,bertini2015macroscopic,derrida2007non} as and understanding the formal gradient flow structures \cite{peletier2014large,adams2011large,adams2013large,duong2013wasserstein} associated to \eqref{eq:heat}. Starting from the formal gradient flow structure unveiled in \cite{bertini2005large}, we address the open question of how far the formal structure governs the large deviations, or whether other types of behaviour are possible. We refer to \cite{fehrman2019large,fehrman2021well,gess2023rescaled,fehrman2025matching} for a discussion of this issue in the context of the zero-range process, as well as \cite{heydecker2023large} for an example where the formal structure does not describe all trajectories. \\

The static large deviations are, by now, known, both in equilibrium \cite[Theorem 1]{peletier2014large} and for the non-equilibrium steady states coupled to heat baths \cite{bertini2005large,de2024hidden}. The starting point for the current work is the analysis of \cite{bertini2005large}, who formally showed that the dynamic cost associated to a trajectory $u$ is bounded above by $\mathcal{J}^{\rm FP}(u):=\frac12\int u(t,x)^2|\nabla H(t,x)|^2 dtdx$, whenever $u$ is a solution to the Fokker-Planck-like equation \begin{equation}\label{eq:FP-intro} \partial_tu(t,x)=\frac12\Delta u(t,x)-\nabla\cdot(u(t,x)^2\nabla H(t,x));\qquad H\in C^{1,2}_{t,x}.\end{equation} This equation stands in contrast to the analogous ones obtained in the contexts of exclusion-type processes \cite{kipnis1989hydrodynamics,quastel1999large} or the zero-range process (ZRP) \cite{benois1995large}, and the two fundamental hurdles become apparent at the level of this equation. \\\\The first difficulty is analytic. In contrast to the ZRP case, where both the diffusivity and mobility are governed by the same function, the diffusivity appearing in \eqref{eq:FP-intro} is linear, while the mobility multiplying the drift is quadratic. This suggests that, at the large deviation level, the fluctuations may compete with, or even overwhelm, the diffusive effects, despite the regularising properties of the macroscopic evolution \eqref{eq:heat}.  \\\\ The second obstruction is probabilistic. Existing proofs of dynamical lower bounds for conservative systems rely on a superexponential replacement principle \cite{kipnis1989hydrodynamics},\cite[Theorem 3.1] {kipnis1998scaling}, and in the context of the KMP model such an estimate is required for quadratic mobility, see also the discussions in \cite[Section 3.3]{bertini2005large}, \cite[Section VII]{peletier2014large}. Neither the exponential tails of the invariant product measures nor the Dirichlet-form methods available in the zero-range literature provide a superexponential estimate of the required strength; indeed, the analysis of \eqref{eq:FP-intro} below suggests that it is false in $d\ge 2$.  \\\\ The goal of the current work is to show that these obstacles are not merely technical hurdles, but rather fundamental. We will show that the balance between diffusivity and mobility, in particular, allows the construction of singular trajectories. As discussed in \cite{de2024hidden,kim2025spectral}, the KMP process is well-defined and physically reasonable in any physical dimension $d\ge 1$, and we therefore work at this level of generality; part of the surprise of the results is that the pathologies which appear are (strongly) dependent on the dimension. \\ \\ The arguments we develop are not unique to the KMP process, and we analyse the KMP process as a representative of similar processes, such as those defined in \cite{peletier2014large}. The key constructions take place already at the level of the limiting trajectories, and so the same conclusions would hold for other models where equations with the same form as either \eqref{eq:FP-intro} or its `skeletonised' equivalent \eqref{eq: sk} appear, such as other models of heat conduction in \cite{peletier2014large} or the stochastic diffusion-advection equation $$ \partial_t \rho=\frac12\Delta \rho-\sqrt{\epsilon}\nabla(\rho\circ \dot{W}^\delta); \qquad \rho\ge 0$$ for a space-time white noise $\dot{W}$, mollified on a spatial scale $\delta>0$, in a suitable joint limit $\epsilon, \delta\to 0$.  
\subsection{The KMP process \& A Putative Rate Function}

We consider the KMP process $\xi^N_t$, as in \cite{kipnis1982heat}, on the discrete torus $\TND = (N^{-1}\Z/\Z)^d$. Informally, at each  $x\in\TND$, we place a non-negative energy variable $\xi(x)\in[0,\infty)$, and on each edge a Poisson clock of rate $N^2$. When a clock on edge $\{x,y\}$ rings, the sites $x,y$ instantaneously thermalise: a sample $P$ is drawn from the uniform distribution on $[0,1]$, and the total energy between sites $x,y$ is redistributed according to $$ \xi(x)\mapsto P(\xi(x)+\xi(y));\qquad \xi(y)\mapsto (1-P)(\xi(x)+\xi(y)). $$
We defer a formal definition to Section \ref{sec:preliminaries}. To each configuration we associate an empirical measure by
\[
\pi^N(\xi) := \frac1{N^d}\sum_{x\in\TND}\xi(x)\,\delta_x.
\]
 The process admits a one-parameter family of invariant (equilibrium) distributions $\nu^N_\rho$, under which each component is independently, exponentially distributed with mean $\rho$. \\ \\ We next introduce the putative rate function on a forever fixed time interval $[0,T]$. For a scalar $\rho$ and a measure $\xi_0$, the static cost $S_\rho$ is given by  \cite[Theorem 1]{peletier2014large}: \begin{equation}\label{eq:srho}
 	S_\rho(\xi_0):=\left\langle \frac1\rho, \xi_0\right\rangle-\int_{\TTd}\log\left(\frac{u_0(x)}{\rho}\right)dx - 1
 \end{equation} where $u_0$ is the density of the absolutely continuous part of $\xi_0$. For the dynamic part of the rate function, it is convenient to rewrite the cost $\mathcal{J}^{\rm FP}$ defined above \eqref{eq:FP-intro} in a more suggestive form, following the logic of \cite{fehrman2019large}. Taking $g:=u\nabla H$, the Fokker-Planck equation \eqref{eq:FP-intro} becomes the skeleton equation \begin{equation} \tag{Sk$_g$}
 	\label{eq: sk}\partial_tu=\frac12\Delta u-\nabla\cdot(ug); \qquad g\in L^2_{t,x}
 \end{equation} and we define the dynamic cost 
\begin{equation}\label{eq:intro-action}
\cJ(u)
:=\frac12 \inf\Big\{ \|g\|_{L^2_{t,x}}^2 : u \text{ solves \eqref{eq: sk}} \Big\}
\end{equation} where the skeleton equation is understood in a weak sense, which is meaningful as soon as $u\in L^2_{t,x}$. The overall rate function is then defined as \begin{equation}
	\label{eq:intro-action-2} \cI_\rho(\xi):=S_\rho(\xi_0)+\cJ(\xi).
\end{equation}
As discussed above, we do not attempt to prove a full large deviation principle with matching bounds; Indeed, the main result Theorem \ref{thm:intro-pathological} shows that the large deviations include trajectories which lack the regularity or integrability to be interpreted as solutions to \eqref{eq: sk}. We consider the trajectories of empirical measures $\pi^N(\xi^N)=(\pi^N(\xi^N_t):0\le t\le T)$ as elements of a topological space $\cX$, defined in Section \ref{relaxed}, and which may be identified with the maps $t\mapsto \xi_t$ of measures with constant total energy. When $\Xi\in \cX$ admits the representation $t\mapsto \xi_t$ for some map which is continuous, respectively right-continuous and left-limited, for the weak topology on space of measures $\cM_+(\TTd)$, we say that $\xi$ is the unique continuous, respectively c{\`a}dl{\`a}g, representative of $\Xi$. In general, an arbitrary $\Xi\in \cX$ need not be of either such form. We seek a lower bound, restricted to the set of regular paths \begin{equation}\label{eq:intro-good-class}
\cR := \cX \cap \Big\{ u\in C^\infty([0,T]\times\TTd): \inf_{(t,x)} u(t,x) >0 \Big\}.
\end{equation}

\subsection{Statement of Results} The first result is the following, which rigorously validates the rate function on sufficiently regular paths. \begin{theorem}[Restricted lower bound]\label{thm:intro-restricted-lb}
Let $(\xi^N_t)_{t\in [0,T]}$ be a KMP process on $\TND$, with initial data $\xi^N_0$ drawn from the equilibrium measure $\nu^N_\rho, \rho\in (0,\infty)$. Then, for every $u\in \cR$ and every open neighbourhood $\cU\ni u$, it holds that
\[
\liminf_{N\to\infty} \frac1{N^d}\log \P\big(\pi^N(\xi^N)\in\cU\big)
\geq -\cI_\rho(u).
\]
Moreover, $\cJ^{\rm FP}=\cJ$ on the set $\cR$. 
\end{theorem}


The main result addresses the question, in light of Theorem \ref{thm:intro-restricted-lb}, of whether a matching upper bound is possible, and hence whether the candidate function $\mathcal{I}$ describes all possible large deviations. The following theorem shows that the large deviations may develop singularities in any dimension $d\ge 1$, with increasingly severe pathologies as the dimension increases. In particular, trajectories obtained in this way need not be continuous in time in the weak topology of measures in $d\ge 2$, and certainly cannot be interpreted as solutions to \eqref{eq: sk}. We write $\inf_{U\ni \Xi}$ for the infimum of a quantity over all open neighbourhoods of a trajectory $\Xi\in \cX$. 

\begin{theorem}[Pathological trajectories as large deviations]\label{thm:intro-pathological}
Let $(\xi^N_t)_{t\in [0,T]}$ be a KMP process on $\TND$, with initial data $\xi^N_0$ drawn from an equilibrium measure $\nu^N_\rho, \rho\in (0,\infty)$. Then:
\begin{enumerate}[label=(\roman*)]
    \item in dimension $d=1$, for any $t_0\in [0,T]$, there exists a path $\Xi\in \cX$ whose unique continuous-in-time representative $\xi$ admits a smooth density $u_t(x)$ for $t\neq t_0$, such that $\xi_{t_0}$ is singular with respect to the Lebesgue measure, and where \begin{equation}
    	\label{eq: collapse to singularity d=1} \inf_{\cU\ni\Xi} \liminf_N \frac1N\log\PP\left(\pi^N(\xi^N)\in \cU\right)>-\infty;
    \end{equation} \item in dimension $d=2$, for any sequence $\epsilon_k>0$ and distinct times $t_k\in (0,T), k\ge 1$ which satisfy $\sum_k \epsilon_k^{\gamma}<\infty$ for some $0<\gamma<1$, there exists a path $\Xi\in \cX$ whose unique c{\`a}dl{\`a}g representative $\xi$ satisfies, for each $k$, $$ \|\xi_{t_k}-\xi_{t_k-}\|_{\rm TV}=2\epsilon_k $$ and such that\begin{equation}
    	\label{eq: collapse to singularity d=2} \inf_{\cU\ni\Xi} \liminf_N \frac1{N^2}\log\PP\left(\pi^N(\xi^N)\in \cU\right)>-\infty;
    \end{equation}
    \item in dimension $d\ge 3$, fix $c,M\in (0,\infty)$. Then there exists $J=J(\rho, c,M)<\infty$ such that, for any relaxed measure $\Xi$ with constant energy $e(\Xi)\le M$ and a lower bound $\Xi(dt,dx)\ge c dt dx$, it holds\begin{equation}
    	\label{eq: pathologies d>=3} \inf_{\cU\ni \Xi}\liminf_N \frac1{N^d}\log \PP\left(\pi^N(\xi^N)\in \cU\right)\ge -J(\rho, c,M).
    \end{equation}\end{enumerate}

\end{theorem}

\subsection{An Overview of the Proof} We first give an overview of the key ideas in Theorem \ref{thm:intro-pathological}. Starting from the skeleton equation \eqref{eq: sk} and formally computing $\frac{d}{dt}\int \psi(u) dx$ for a smooth nonlinearity $\psi$ , we find \begin{equation}\begin{split}
	\frac{d}{dt}\int_{\TTd} \psi(u(t,x)) dx &= -\frac12\int_{\TTd}\psi''(u(t,x))|\nabla u(t,x)|^2 dx \\ & + \int_{\TTd}(\psi''(u(t,x))u(t,x)) \nabla u(t,x)\cdot g(t,x) dx. \end{split}
\end{equation} Since the only quantity controlled by the rate function is $\|g\|_{L^2_{t,x}}$, this is the only structure to which we may appeal: Any other structure may be lost along a sequence of smooth paths with uniformly bounded rate. Using Cauchy-Schwarz to extract $\|g(t)\|_{L^2_x}^2$ from the second term, we find that the estimate closes if, and only if, the nonlinearity $\psi$ satisfies $$ (\psi''(u))^2u^2\lesssim \psi''(u) $$ which forces the choice $\psi''(u)=u^{-2}$, $\psi(u)=-\log u$. The stable estimate may therefore be written \begin{equation}\label{eq: entropy est}
	S_\rho(u_T)+\frac12\|\nabla \log u\|_{L^2_{t,x}}^2 \le S_\rho(u_0)+\frac12\|g\|^2_{L^2_{t,x}}
\end{equation} for the function $S_\rho$ given by \eqref{eq:srho}. In contrast to the corresponding estimates in \cite{FG22}, the static entropy $S_\rho$ grows asymptotically linearly in the density, and so $S_\rho(u)$ may remain finite while $u$ collapses to a singular measure. \\ \\ The dependence on dimension enters through what sort of compactness or boundedness may be obtained from the integrated entropy dissipation $\|\nabla \log u\|_{L^2_{t,x}}^2$. In dimension $d=1$, the Sobolev embedding and the finiteness of the (constant) energy is sufficient to guarantee that $u\in L^2_{t,x}$, and the equation \eqref{eq: sk} can a priori be understood in a weak form. In $d=2$, the Moser-Trudinger inequality shows that the sublevel sets $$\{u\in L^1(\TT^2,[0,\infty): \|\nabla \log u\|_{L^2_x}^2\le M, \langle u, 1\rangle \le M\} \subset L^1(\TT^2)$$ are compact in the weak topology of $L^1_x$, and so the finiteness of the integrated entropy dissipation $\|\nabla \log u\|_{L^2_{t,x}}^2$ forces $u$ to be a $L^1_x$ function for $dt$-almost all times $t$. However, the embedding is no longer strong enough to force $u\in L^2_{t,x}$, so that the product $ug$ is not well-defined. We exploit this by constructing $u^n\in \cR, g^n$ solving \eqref{eq: sk}, where $g^n$ are bounded in $L^2_{t,x}$, and where, over asymptotically vanishing time slices $I$, $\int_I \int_{\TT^2}|g^n|^2\to 0$ while $\int_I \int_{\TT^2} u^n g^n\not \to 0$, producing instantaneous jumps in the limit. \\ \\ In dimension $d\ge 3$, the picture changes again, because the instantaneous entropy dissipation is no longer even lower semicontinuous for the weak topology of $L^1_x$. For this reason, it is possible to approximate an arbitrary measure $\Xi$, with the same sort of assumptions $\Xi(dx)\ge cdx, \langle 1, \Xi\rangle \le M$ as in the theorem by $u^n$ with $S_\rho(u^n)\le C(\rho,c,M), \|\nabla \log u^n\|_{L^2_x}\to 0$. Each $u^n$ is a superposition of many concentrated spikes, and we will see that the dynamic cost of rapidly moving the spikes may also be constructed to vanish in the limit.

\subsection{Organisation of the Paper}

Section~\ref{sec:preliminaries} gathers notation used repeatedly later, including the formal definitions of the KMP process and the topological space $\cX$ of relaxed paths on which the LDP takes place. Section~\ref{sec:restricted-lb} contains the restricted lower bound and is written so as to isolate exactly where the properties of the constructed change of measure $\QQ^{N,K}$ are used to obtain a replacement estimate without a superexponential estimate. Section~\ref{sec:pathological} constructs the pathological paths by from the lower semicontinuous envelope of $\mathcal{I}|_{\cR}$.

\subsection{Literature Review \& Discussion}\label{sec: LRD}

\subsubsection{The KMP Process and Related Models}

The KMP process considered here was introduced in \cite{kipnis1982heat} as a stochastic model of heat transport for which Fourier's law and the linear stationary temperature profile could be derived rigorously. The hydrodynamic and Fourier-law aspects for related models were further developed by \cite{bernardin2005fourier,bernardin2007hydrodynamics,bernardin2008stationary,giacomin1999deterministic}. We refer also to \cite{bertini2005large,bertini2015macroscopic} for a macroscopic fluctuation theory (MFT) viewpoint, in which the KMP process has the distinction that mobility $\chi(\rho)$ is quadratic rather than linear. We also refer to \cite{carinci2013duality,de2024hidden,carinci2024solvable} for derivation of the MFT predictions, as well as a rigorous derivation of the quasipotential \cite{de2024hidden} via a hidden temperature. The exposition of the present work in the periodic, equilibrium setting, in contrast to the most physically important setting of non-equilibrium with fixed temperature reservoirs at the boundary, is primarily made to ease exposition and notation. Indeed, the driving mechanism of Theorem \ref{thm:intro-pathological}, namely the balance between linear diffusivity and quadratic mobility, remains the same with either periodic or Dirichlet boundary conditions.
A body of work has also analysed more general energy-transport models such as the Brownian energy process (BEP$(m)$) \cite{giardina2009duality,carinci2013duality} and `stick-breaking' processes \cite{feng1997microscopic,suzuki1993hydrodynamic}. \cite{peletier2014large} proved the hydrodynamic limit for $\mathrm{BEP}(m)$ and gave formal derivations of the pathwise large-deviation functional and associated gradient flow structures for $\mathrm{BEP}(m)$, KMP and the generalised Brownian Energy Process. The works \cite{bernardin2007hydrodynamics} and \cite{feng1997microscopic,suzuki1993hydrodynamic} derive the heat and porous medium equations respectively as the hydrodynamic limit, but rely on $L^p_{t,x}$ estimates which are generally not possible at the large deviations scale.

\subsubsection{Large Deviations in Lattice Systems \& The Superexponential Replacement Estimate}

The probabilistic strategy behind Theorem~\ref{thm:intro-restricted-lb} belongs to the now classical large-deviation method for interacting particle systems \cite{kipnis1989hydrodynamics,donsker1989large,varadhan1997diffusive,quastel1999large,kipnis1998scaling}. On the lower-bound side, one introduces a weakly tilted dynamics, proves tightness under the tilted law, identifies every subsequential hydrodynamic limit as a weak solution of the desired skeleton equation, and then computes the relative entropy cost of the change of measure. The most involved step is the validity of the (superexponential) replacement lemma, where one must show that, for every $\delta>0$ and every regular test function $\varphi$,
\begin{equation}\label{eq: superexponential prototype}
 \limsup_{\epsilon\to 0}\limsup_N \frac1{N^d}\log \mathbb P\Bigg(\Big|\int_0^T V^{N,F,\varphi}_{\epsilon}(t, \xi^N_t)\,dt\Big|>\delta\Bigg)=-\infty
\end{equation}
where $F$ is a local function of the configuration, and $V^{N,F,\varphi}_\epsilon$ is an error given by replacing $F(\xi^N)$ by functions of the local average $\bar{\xi}^{N\epsilon}$ on a small macroscopic scale, see Sections \ref{subsec:technical} and \ref{subsec:proof-technical-inputs}. In the settings of exclusion and zero-range processes, such estimates are proven in \cite{kipnis1989hydrodynamics,donsker1989large,varadhan1997diffusive,quastel1999large,kipnis1998scaling}, and the same arguments can readily be seen to apply to the KMP process without essential change \cite{bertini2025private}, as soon as $F$ is bounded and Lipschitz-continuous. In contrast, as discussed in \cite{bertini2005large,peletier2014large}, the quadratic mobility inherent to the KMP model does not fit into this class, nor can the observables be truncated using the tails of the invariant distributions $\nu_\rho$. We refer also to \cite{gess2023rescaled,gess2026porous} for a similar issue in the context of the zero-range process with superlinear jump rates.\\

 We do not attempt to resolve the issue of the superexponential estimate at the level necessary to take quadratically growing $F$ in \eqref{eq: superexponential prototype}. Instead, part of the contribution is a method to bypass this obstruction in the context of a restricted lower bound. Provided that the desired limit is sufficiently regular - in the present context, $u\in\mathcal R$ - it is possible to construct candidate changes of measure $\QQ^N\ll \PP$ in order to exploit the regularity of $u$ to deduce additional integrability for $\xi^N$ in $\QQ^N$-expectation. With this in hand, and using the validity of \eqref{eq: superexponential prototype} for bounded and Lipschitz $F$, the error in \eqref{eq: superexponential prototype} may be shown to be small with high $\QQ^N$-probability, which is sufficient for the lower bound. Crucially, we seek information only about the specially constructed $\QQ^N$, and not about generic changes of measure at the exponential scale.
 
 We further remark that the dependence on dimension appears to be a somewhat different effect from that investigated in \cite{quastel1999large}, see also the discussion in \cite{QY}. In these works, the dimension appears in determining the correction to the diffusivity, respectively logarithmic super-diffusivity in $d=2$. In contrast, both the diffusivity and mobility for our model are fixed, independent of dimension, and the dimension-dependence enters via the strength of the available Sobolev embeddings, as discussed under \eqref{eq: entropy est}. 
\subsubsection{Parabolic-Hyperbolic PDE with Rough Drift}

The examples in Theorem \ref{thm:intro-pathological} are made possible by exploiting the link between large deviations and parabolic-hyperbolic PDE with rough drift \cite{giacomin1999deterministic,fehrman2019large,dirr2020conservative,fehrman2021well,fehrman2025matching}, which motivates the reformulation of the Fokker-Planck-like equation \eqref{eq:FP-intro} into \eqref{eq: sk}. As in the analysis leading to \eqref{eq: entropy est}, while each $u \in \cR$ and associated $H$ have better regularity, any regularity beyond that guaranteed by the rougher form \eqref{eq: sk} is lost in passing to the lower semicontinuous envelope.
The analysis of such equations goes back to the renormalised solution theory \cite{diperna1989ordinary,ambrosio2004transport}, see also \cite{crippa2008estimates,de2008ordinary,jabin2010differential,jabin2016critical}. Solution theories of kinetic solutions were developed in \cite{LPT94,P02}, and the notion of entropy of solutions was developed in \cite{chen2003well,karlsen2003uniqueness,bendahmane2004renormalized,lellis2003structure}. We remark that the formulation \eqref{eq: sk} forces $g\in L^2_{t,x}$, which is much less regular than the cases considered in these works; for instance, \cite{karlsen2003uniqueness} requires $g\in L^1_t W^{1,1}_x\cap L^\infty_tC^0_x$. The construction of the trajectories in Theorem \ref{thm:intro-pathological} in Section \ref{sec:pathological} may be understood as exploiting ill-posedness phenomena for \eqref{eq: sk}, for which we refer to  \cite{modena2018non,modena2020convex}. As before, the setting relevant to the KMP process is much rougher than that of the cited works, allowing for a more dramatic failure of regularity.\section{Preliminaries}\label{sec:preliminaries}

\subsection{Discrete Torus, Empirical Measures, Function Spaces}

We work with functions on the discrete torus $\TND$, which we define including the rescaling as $\TND:=\{0,N^{-1}, \dots, 1-N^{-1}\}^d$. We write $x\sim y$ for adjacency of $x,y\in \TND$ as nearest neighbours and, to avoid overcounting certain orientation-sensitive quantities, fix an orientation of each pair of nearest neighbours. In the sequel, the sum $\sum_{x\sim y}$ will therefore stand for a sum over all edges $E(\TND)$, each counted once with its privileged orientation. One may check that the choice of orientations plays no r\^ole in determining `physical' quantities, such as the macroscopic evolution, any may be chosen arbitrarily. 
 
For $f,g:\TND\to\RR$ we write
\[
\ip{f,g}_N := \frac1{N^d}\sum_{x\in\TND} f(x)g(x),
\qquad
\norm{f}_{L^p_N} := \Big(\frac1{N^d}\sum_{x\in\TND} \abs{f(x)}^p\Big)^{1/p}.
\]
Throughout, we use the abbreviations $\|\cdot\|_{L^p_x}, \|\cdot\|_{L^p_{t,x}}$ for the usual Lebesgue norms for functions of space, respectively time and space, under the usual Lebesgue measure $dx$, respectively $dt dx$.

\subsection{Formalities of the KMP Process}

We now give some formalities of the KMP dynamics defined in \cite{kipnis1982heat}. The process is defined on the state space $\Lambda_N:=[0,\infty)^{\TND}$ with generator \begin{equation}\label{eq:intro-generator}
\cL_t^{N} F(\xi) = N^2\sum_{x\sim y} L_{x,y}F(\xi)=N^2\sum_{x\sim y}\int_0^1 [F(\xi^{x,y,p})-F(\xi)] dp
\end{equation}
where the new configuration is given by redistributing the energy between $x$ and $y$:
\begin{equation}\label{eq:reconfig}
\xi^{x,y,p}(z)=\begin{cases}\xi(z) & \text{for } z\notin\{x,y\} \\
p(\xi(x)+\xi(y)) & z=x; \\ (1-p)(\xi(x)+\xi(y)) & z=y. \end{cases}
\end{equation}


\paragraph{\bf Empirical Measures}To each microscopic energy configuration, we associate the empirical measure \begin{equation}\label{eq:empirical-measure-prelim}
\pi^N(\xi):=\frac1{N^d}\sum_{x\in\TND}\xi(x)\,\delta_x\in \cM_+(\TTd).
\end{equation}
We will also use a restriction of the state space $\Lambda_{N,a}$ of configurations with total energy bounded above by $a$: 
\[
\Lambda_{N,a}:=\{\xi\in [0,\infty)^{\TND}: \langle 1, \pi^N(\xi)\rangle \le a\}.
\]
Since the KMP redistribution mechanism \eqref{eq:reconfig} preserves the total energy $\langle 1, \pi^N(\xi)\rangle=\ip{\xi}_N$ globally, each $\Lambda_{N,a}$ is invariant for the dynamics specified by \eqref{eq:intro-generator}. 

\paragraph{\bf Empirical Flux} It is convenient, for later application, to enhance the dynamics with an auxiliary `flux measure' $\mu^N$. We define $\mu^N$ to be the unnormalised empirical measure on $[0,T]\times E(\TND)\times [0,1]$, assigning mass $1$ to each tuple $(t,(x,y),p)$ such that, at time $t$, energy is redistributed across the (oriented) edge $(x,y)$ according to \eqref{eq:reconfig} with parameter $p$. $\mu^N$ is thus a Poisson random measure with intensity \begin{equation}
	\label{eq:prm-intensity} \bar{\mu}^N(dt,(x,y),dp)=N^2dp dt.
\end{equation}

\paragraph{\bf Invariant Measures} It is well-known \cite{peletier2014large} and can be straightforwardly verified that the KMP dynamics are ergodic on each simplex $$ \partial \Lambda_{N,a}=\{\xi\in[0,\infty)^{\TND}: \langle 1, \pi_N(\xi^N)\rangle=a\}.$$ Each microcanonical ensemble, given by the uniform measure on $\partial \Lambda_{N,a}$, is thus invariant for the dynamics. We will work with the canonical invariant measures $\nu^N_\rho, \rho\in (0,\infty)$ under which the energy of each site is an independent exponential random variable: \begin{equation}
	\label{eq:equilib-rho} \nu^N_\rho(d\xi^N)=\prod_{x\in \TND} \frac1\rho e^{-\xi^N(x)/\rho}d\xi^N(x).
\end{equation} Similarly, for a profile $u\in C(\TTd,[0,\infty))$, we denote $\nu^N_u$ the slowly varying local equilibrium \begin{equation}
	\label{eq:SVLE} \nu^N_u(d\xi^N)=\prod_{x\in \TND} \frac1{u(x)}e^{-\xi^N(x)/u(x)}d\xi^N(x).
\end{equation} In the case where $u$ vanishes on one or more sites, the corresponding marginals in \eqref{eq:SVLE} should be understood to be the degenerate distribution $\delta_0$. We will always use $u$ for spatially dependent profiles, in order to distinguish between global and slowly varying local equilibria. We will use the following lemma. \begin{lemma}[Convergence under slowly varying local equilibrium] \label{lem:svle}Let $u\in C(\TTd,[0,\infty))$. Then the law of $\pi_N(\xi^N_0)$ under $\xi^N_0\sim \nu^N_u$ converges in distribution to $\delta_u$. The same holds for the measures $\nu^N_{u, a}$  \begin{equation}
	\nu^N_{u,a}(A):=\frac{\nu^N_u(A\cap \{\xi^N\in \Lambda_N: \langle 1, \pi_N(\xi^N)\rangle \le a\}}{\nu^N_u(\{\xi^N\in \Lambda_N: \langle 1, \pi_N(\xi^N)\rangle \le a\})}; \quad A\subset \Lambda_N \end{equation} which are the measures of $\nu^N_u$ conditioned to have total mass not exceeding $a$, as soon as $a>\langle 1, u\rangle$. 
	
\end{lemma}

\subsection{The Path-Space of Relaxed Measures} \label{relaxed}

We next specify a topological space in which we consider the large deviations. Throughout, for a compact topological space $X$,
$\cM_+(X)$ denotes the space of finite nonnegative Borel measures on $X$, and $\cM_a(X)$ for those measures with mass at most $a$.
\begin{definition}[Relaxed-measure path space]
We define $\mathcal X$ to be the set of all finite nonnegative Borel measures $\Xi$ on
$[0,T]\times \TTd$ whose time marginal is of the form \[
\Xi(dt\times\TTd)=e(\Xi)dt
\] for some $e(\Xi)\in [0,\infty)$. We equip $\mathcal X$ with the topology of weak convergence of measures on $[0,T]\times\TTd$, that is, the topology induced by the maps\[
\Xi\mapsto \int_{[0,T]\times\TTd} \phi(t,x)\,\Xi(\dd t,\dd x)
\]
for continuous test functions $\phi$ on $[0,T]\times\TTd$.
\end{definition}
Similarly to the state spaces, we write $\cX_a$ for the set of those $\Xi\in \cX$ where $e(\Xi)\le a$, each of which is compact by Prokhorov's theorem.\\\\ We also remark on some useful identifications in order to ease notation. Since the time-marginal of any $\Xi \in \cX$ is absolutely continuous with respect to the Lebesgue $dt$, it may be identified with its disintegration $(\xi_t)_{t\in[0,T]}$, such that
$\xi_t(\TTd)$ is constant in time. In the sequel, we identify any measurable map $\xi: [0,T]\to \cM_+(\TTd)$ with constant mass to the relaxed path $\Xi(dt,dx)=\xi_t(dx)dt$, identifying trajectories which coincide for almost all times. Similarly, to ease notation, we do not distinguish between a density $u\in L^1_x, u\ge 0$ and the measure it induces, nor between a nonnegative function $u\in L^\infty_tL^1_x$ and the corresponding relaxed path. \\ \paragraph{\bf Distances on Measures} It is convenient to introduce some distances on measures and relaxed measures which induce, on bounded sets, weak convergence or convergence in $\cX$. For two finite measures $\mu, \nu$ on $\TTd$, we set \begin{equation}
	\label{eq: def W} W(\mu, \nu):=\|\mu-\nu\|_{W^{-1,1}(\TTd)}=\sup\{\langle f, \mu-\nu\rangle: \|f\|_{W^{1,\infty}(\TTd)}\le 1\}.
\end{equation} We also write $\|\cdot\|_{\rm TV}$ for the total variation norm. For relaxed measures $\Xi, \Xi'\in \cX$, we write $d(\Xi, \Xi')$ for the distance induced by duality against functions which are Lipschitz in both space and time: \begin{equation}
	d(\Xi, \Xi'):=\sup\left\{\langle f, \Xi-\Xi'\rangle: \|f\|_{L^\infty_t W^{1,\infty}_x\cap W^{1,\infty}_tL^\infty_x}\le 1\right\}.
\end{equation} \paragraph{\bf Continuous and C{\`a}dl{\`a}g Trajectories} Although it is not possible to restrict the LDP to trajectories enjoying any continuity properties, it is convenient to introduce notation and terminology. We say that a map $t\mapsto \xi_t$ is continuous and write $\xi \in C([0,T],\cM_+(\TTd))$ if it is continuous when $\cM$ is equipped with the distance $W$. Similarly, we say that $\xi$ is c{\`a}dl{\`a}g and write $\xi \in D([0,T],\cM_+(\TTd))$ if the map $t\mapsto \xi_t$ is right-continuous in the distance $W$, and $\xi$ admits left-limits in $W$: $$ W(\xi_s,\xi_{t-})\to 0 \text{ as }s\uparrow t.$$ We equip the spaces $C([0,T], \cM_+(\TTd)), D([0,T], \cM_+(\TTd))$ with the topologies of uniform, respectively Skorokhod, convergence for the metric $W$. \\ \\Using the identifications above, we will shorten referring to the existence of continuous or c{\`a}dl{\`a}g representatives of $\Xi\in \cX$ to saying that $\Xi$ itself is continuous, respectively c{\`a}dl{\`a}g. If such a representative exists, it is necessarily unique.

\section{Restricted Lower Bound without a Superexponential Estimate}\label{sec:restricted-lb}

In this section we give the proof of Theorem~\ref{thm:intro-restricted-lb}. We therefore fix, forever, $u\in \cR$, and $\rho>0$ as in the theorem. We first give an overview of the proof; several compensations are needed to deal both with the potentially poor behaviour of the skeleton equation appearing in \eqref{eq:intro-action}, and with the lack of a suitable superexponential estimate for the mobility. \\ \\ The argument proceeds by the usual `tilting' or relative entropy method \cite[Lemma 7]{mariani2010large}: it is sufficient to exhibit a change of measures $\QQ^N\ll\PP$ such that, under $\QQ^N$, $\pi_N(\xi^N)\to u$, with relative entropy $$ \limsup_N \frac1{N^d}{\rm Ent}(\QQ^N|\PP)\le \cI(u).$$ Formally, the desired transformation $\QQ^N$ was written down in \cite[Section 3.2]{bertini2005large}. Both difficulties discussed in the introduction appear in making this argument rigorous. First, since the exponential martingale defining $\frac{d\QQ^N}{d\PP}$ involves local functionals of the form $\sim \xi(x)^2$, proving that $\pi_N(\xi^N)$ concentrates on solutions to the Fokker-Planck equation \eqref{eq:FP-intro} requires the superexponential estimate for this mobility. Secondly, due to the lack of superexponential tightness, the limit path may a priori only be measure-valued, and one cannot (na\"ively) make sense of, or prove well-posedness for, the limiting equation. The resolution is as follows. \begin{enumerate} \item First, we write down, for general $H$ and cutoffs $K\in (0,\infty)$, the same change-of-measures in \eqref{eq:RND-tilt-1}-\eqref{eq:RND-tilt-3}, but with all occurrences of $\xi(x)$ in the Radon-Nikodym derivative replaced by $\xi(x)\land K$. \item Secondly, instead of seeking replacement lemmas for the functions of the form $\xi(x)(\xi(y)\land K)$ which are superexponential, and would therefore hold for {\em any} change of measures, we instead tailor the replacement lemma to $\QQ^{N,K}$, allowing us to exploit the regularity of the desired $u\in \cR$. This is achieved by using the boundedness of the drift to prove, for each fixed $K$, an integrability estimate for $\|\xi\|_{L^2_{t,x}}$ Together with the validity of the superexponential replacement lemma for bounded local functions, this allows us to prove the replacement lemma under $\QQ^{N,K}$. \item The limiting equation, for the change of measure thus constructed, takes the form \begin{equation}
	\label{eq: FP-cutoff-prototype} \partial_t v =\frac12\Delta v -\nabla \cdot (\chi_K \Theta_K(v)\nabla H)
\end{equation} where $\Theta_K$ is Lipschitz continuous, and hence admits an extension to measures. Using Lipschitz continuity, it is possible to prove uniqueness for equations of the form \eqref{eq: FP-cutoff-prototype} in the space of measure-valued trajectories. \item Finally, instead of viewing a limiting $u\in \cR$ as a solution to \eqref{eq:FP-intro}, the function $\chi_K$ may be chosen so that the desired $u$ is a solution to \eqref{eq: FP-cutoff-prototype}. The relative entropy associated to the change of measures is no longer exactly $\frac12\int u^2|\nabla H|^2$, as in \cite[Section 3.2]{bertini2005large}, but disagrees only by a multiplicative factor which $\to 1$ as $K\to \infty$. \end{enumerate}\subsection{Properties of $g$}

\begin{lemma}[Regularity of the optimal drift]\label{lem:optimal-drift-smooth}
For any $u\in\cR$, there exists a unique $g\in L^2_{t,x}$ attaining the minimum \eqref{eq:intro-action}, which additionally is of the form $g=u\nabla H, \nabla H\in C^\infty_{t,x}$, and hence is smooth $g\in C^\infty([0,T]\times\TTd,\RRd).$ 
\end{lemma}

\begin{proof}
Fix $u$. First, we note that the set of candidate $g$
\[
\mathfrak C:=\Big\{g\in L^2([0,T]\times\TTd;\RR^d): \nabla\cdot(ug)=\frac12\Delta u-\partial_t u\Big\}
\]
is nonempty, because it contains 
\[
g_0:=u^{-1}\left(\frac12\nabla u+\nabla(-\Delta)^{-1}\partial_t u\right)
\]
which is well-defined in $C^\infty_{t,x}\subset L^2_{t,x}$ thanks to the smoothness and strict positivity of $u$, and because $\partial_tu$ is a smooth function with zero mean on each time-slice. Next, by the definition of weak solutions, $\mathfrak C$ is a closed, affine subspace of the Hilbert space $L^2([0,T]\times\TTd;\RR^d)$, and since the norm $\|\cdot\|_{L^2_{t,x}}$ is strictly convex, the problem \eqref{eq:intro-action} admits a unique minimiser $g$.

We now show that the optimal $g$ must enjoy additional regularity. Let
\[
\mathfrak C_0:=\{h\in L^2([0,T]\times\TTd;\RR^d): \nabla\cdot(uh)=0\}
\] be the space of tangent directions to $\mathfrak C$. By considering the first variation, $\langle g, h\rangle=0$ for all $h\in \mathfrak C_0$, which in turn implies that $g$ belongs to the closure $$ g\in \overline{\{u\nabla H:H\in C^1([0,T]\times\TTd)\}}^{L^2_{t,x}}.   $$ Since $u$ is bounded, this set coincides exactly with $\{u\nabla H: H\in L^2_tH^1_x\}$, and hence $g$ may be written in this form. Making this substitution, the skeleton equation is converted back to the Fokker-Planck form:
\begin{equation}\label{eq:weighted-elliptic-lambda}
\nabla\cdot\big(u^2\nabla H\big)=-\partial_t u+\frac12\Delta u.
\end{equation}
Viewed as an equation for $H$, the smoothness and strict positivity of $u$ ensure that \eqref{eq:weighted-elliptic-lambda} is an elliptic equation with smooth coefficients. Standard
elliptic regularity (see, for example, \cite{ladyzhenskaia1968linear}) therefore implies that $H$ is smooth modulo time-varying additive constants, and hence so are $\nabla H$ and $g=u\nabla H$.\end{proof}
\subsection{Cutoff at Large Energies} Before defining a change of measure, we first specify a cutoff parameter $K$ at large energies and associated notation. For any $K>0$, we set $\tau_K(r):=r\wedge K$. Writing $\xi$ for an exponential variable of mean $1$, we define the truncated moments $m_1^K$, $m_{1,1}^K$, and $m_2^K$ by
\[
m_1^K(\rho):=\E[(\rho \xi)\wedge K],\quad m_{1,1}^K(\rho):=\E[(\rho\xi)(\rho\xi\wedge K)],\quad m_2^K(\rho):=\E[(\rho\xi\wedge K)^2],
\]
so that
\[
m_1^K(\rho)=\rho\big(1-e^{-K/\rho}\big),\qquad m_{1,1}^K(\rho)=2\rho^2-(K\rho+2\rho^2)e^{-K/\rho},
\]
\[
m_2^K(\rho)=2\rho^2-2(K\rho+\rho^2)e^{-K/\rho}.
\]
We further define the combinations
\begin{equation}\label{eq:def-Theta-Gamma-K}
\Theta_K(\rho):=\frac13\big(2m_{1,1}^K(\rho)-\rho m_1^K(\rho)\big),
\qquad
\Gamma_K(\rho):=\frac23 m_2^K(\rho)-\frac13\big(m_1^K(\rho)\big)^2,
\end{equation} as well as the ratios
\begin{equation}\label{eq:def-AK-RK}
A_K(\rho):=\frac{\rho^2}{\Theta_K(\rho)},
\qquad
R_K(\rho):=\frac{\Gamma_K(\rho)\rho^2}{\Theta_K(\rho)^2}.
\end{equation}
Since the path $u$ is smooth and strictly positive, we may choose $0<m\le M<\infty$ so that $m\le u(t,x)\le M$. Using the uniform convergence $\Theta_K(\rho)\to \rho^2$ and $\Gamma_K(\rho)\to \rho^2$ on $[m,M]$, we may choose $K$ large enough that
\begin{equation}\label{eq:choice-of-K}
\sup_{\rho\in[m,M]} |R_K(\rho)-1|
\le {\delta}\left(1+\mathcal{J}(u)\right)^{-1}.
\end{equation}
For the remainder of the section this cutoff level $K$ is fixed, and we abbreviate $\chi_K(t,x):=A_K(u(t,x))$.
\subsection{Construction of a Change of Measure}
We are now in a position to define a change of measure $\QQ^{N,K}$. We fix, for the remainder of this section, $u\in \cR$, write $g=u\nabla H$ with $H\in C^\infty([0,T]\times\TTd)$ as in Lemma \ref{lem:optimal-drift-smooth}, and fix $\delta>0$. 
We first fix $a>\langle 1, u_0\rangle$, and define \begin{equation}\begin{split}
\label{eq:RND-tilt-1} &Y^N_0:=\exp\left(-\sum_{x\in \TND} \left[\frac{\xi^N_0(x)}{u_0(x)}-\frac{\xi^N_0(x)}{\rho}+\log\left(\frac{u_0(x)}{\rho}\right)\right]\right)\\ &\hspace{7cm}\times \frac{1(\langle 1, \pi_N(\xi^N_0)\rangle \le a)}{\nu^N_{u_0}(\langle 1, \pi_N(\xi^N_0)\rangle \le a)} \end{split}
\end{equation} which is readily seen to have $\mathbb{E}[Y_0]=1, Y_0\ge 0$. Next we define, for $x\sim y$ and $p\in[0,1]$,
\begin{equation}\label{eq:def-beta-r}
{\vartheta^N_t(x,y,p,\xi):=\chi_K\!\left(t,m(x,y)\right)\big(H(t,x)-H(t,y)\big)\Big(p\tau_K(\xi(y))-(1-p)\tau_K(\xi(x))\Big),}
\end{equation}
\begin{equation}\label{eq:def-r-cutoff}
{r_t(x,y,p,\xi):=\exp\big(\vartheta^N_t(x,y,p,\xi)\big),}
\end{equation}
where $m(x,y)\in \TTd$ is the midpoint of the edge $x,y$ and, recalling the definition of $\mu^N$ above \eqref{eq:prm-intensity}, the corresponding exponential martingale \begin{equation}\label{eq:RND-tilt-2}\begin{split}
&Z^N_t
:=\exp\Bigg\{
\int_{(0,t]\times[0,1]}\hspace{0.1cm}\sum_{x\sim y} \log r_s(x,y,p,\xi^N_{s-})\,\mu^N(\dd s,(x,y),\dd p)
\\ &\hspace{4.2cm} -N^2\int_{(0,t]\times[0,1]}\!\sum_{x\sim y}\big(r_s(x,y,p,\xi^N_{s-})-1\big)\,\dd p\,\dd s
\Bigg\}.
\end{split} \end{equation} It is readily seen that $Z_t$ is a mean-1, nonnegative $\PP$-martingale, which allows us to define a change of measure by \begin{equation}
\label{eq:RND-tilt-3} \frac{d\QQ^{N,K}}{d\PP}:=Y^N_0 Z^N_T.
\end{equation} Using the Girsanov theorem for jump processes, we obtain the following properties of $\mu^N, \xi^N$ under $\QQ^{N,K}$, which follow immediately from \cite[Appendix 1, Proposition 7.3]{kipnis1998scaling}.
\begin{lemma}[Girsanov tilt]\label{lem:girsanov}
Under the probability measures $\QQ^{N,K}$ given by \eqref{eq:RND-tilt-3}, the processes $\xi^N_t$ are Markov with generator given by \begin{equation}\label{eq:tilted-generator-1}
{\cL_t^{N,H,K} = N^2\sum_{x\sim y} L^{H,K}_{x,y,t};}
\end{equation} \begin{equation}\label{eq:tilted-generator-2}
{L^{H,K}_{x,y,t}F(\xi)
=\int_0^1 r_t(x,y,p,\xi)\big[F(\xi^{x,y,p})-F(\xi)\big] \,\dd p.}
\end{equation} Moreover, under $\QQ^{N,K}$, the empirical jump measure $\mu^N$ has predictable intensity
\[
{N^2 r_t(x,y,p,\xi^N_{t-})\,\dd p\,\dd t,}
\]
and the initial condition is distributed $$\xi^N_0\sim \nu^N_{u_0, a} $$ as a slowly varying local equilibrium with mass conditioned not to exceed $a$, as in Lemma \ref{lem:svle}.
\end{lemma}

We record the following easy estimate for later convenience, see also Lemma \ref{lem:entropy-est} for a refined version of the same argument. \begin{lemma}\label{lem:ent-crude}\begin{equation}\label{eq:entropy-prebound-QNK}
    \limsup_{N\to\infty}
    \frac1{N^d}\Ent(\QQ^{N,K}|\PP)
    \le C_K .
\end{equation} \end{lemma}\begin{proof}
By definition, \begin{equation}\Ent(\QQ^{N,K}|\PP)=\EE_{\QQ^{N,K}}[\log Y^N_0]+\EE_{\QQ^{N,K}}[\log Z^N_T]. \end{equation} Since $u_0$ is strictly positive and the mass is $\QQ^{N,K}$-almost surely bounded, the first term is at most $$ \EE_{\QQ^{N,K}}[\log Y^N_0] \le N^d\left(a\left\|\frac{1}{u_0}-\frac1\rho\right\|_{L^\infty_x}  + \left\|\log\frac{u_0}{\rho}\right\|_{L^\infty_x}\right)$$ which is of the desired order. For the dynamic part, using the $\QQ^{N,K}$-predictable intensity identified in Lemma \ref{lem:girsanov}, the contribution to \eqref{eq:entropy-prebound-QNK} is $$\EE_{\QQ^{N,K}}[\log Z^N_T]=\EE_{\QQ^{N,K}}\left[N^2\int_{(0,T]\times[0,1]}\sum_{x\sim y} (r_s\log r_s-r_s+1)(x,y,p,\xi^N_s) \dd p \dd s\right]. $$ As in \eqref{eq:theta-small}, $
    |\vartheta_t^N(x,y,p,\xi)|\le C_{H,K}N^{-1}
$, and so the integrand is of order
\[
    r_t\log r_t-r_t+1\le C_{H,K}N^{-2}
\]
uniformly in $t,\{x,y\},p$ and $\xi$. Summing over the $dN^d$ edges $x\sim y$ produces the result. \end{proof}

{\subsection{Technical Inputs} \label{subsec:technical} We next give the statements of two technical lemmata for the probability measures $\QQ^{N,K}$ which will be needed in the proofs. To ease readability, the proofs are deferred until Subsection~\ref{subsec:proof-technical-inputs}. We begin with a replacement lemma. \\\\We say that a function $F:\Lambda_N\to\RR$ is {\em local} if it only depends on the values of  $\xi(x), x\in A$ for a fixed, finite set $A$, and (locally) Lipschitz if its dependence in those coordinates is (locally) Lipschitz continuous $F:[0,\infty)^A\to \RR$. In the applications below we shall only use local Lipschitz functions with at most linear growth in these coordinates. For such a function, we define $\bar{F}(\rho)$ by its expectation $$\bar{F}(\rho):=\EE_{\nu^N_\rho}[F(\xi)]$$ under an equilibrium measure. For a configuration $\xi^N$, we define the local average $$\bar{\xi}^{N\eps}(x):=\frac{1}{(2\lfloor N\eps\rfloor+1)^d}\sum_{y\in \mathcal{B}(x,N\epsilon)}\xi(y)$$ where $\mathcal{B}(x,N\epsilon)$ is the lattice box of side length $2\lfloor N\epsilon\rfloor +1$ centred at $x$. For $y\in \TND$, we define $\tau_y:\Lambda_N\to\Lambda_N$ to be the translation $\tau_y\xi^N(x):=\xi^N(x+y)$, and $\tau_yF:=F\circ \tau_y$. For a test function $\varphi \in C([0,T]\times\TTd)$, we define \begin{equation} \label{eq: def V} V^{F,N,\varphi}_\eps(t,\xi):=\frac1{N^d}\sum_{x\in\TND}
\varphi(t,x)
\big[
\tau_xF(\xi_t^N)-\bar F(\bar\xi^{N\eps}_t(x))
\big]. \end{equation} With these defined, we give the statement of the replacement lemma.  \begin{proposition}[Replacement Estimate]\label{prop:replacement} Let $F$ be local, locally Lipschitz and satisfy the growth bound $|F(\xi)|\le C(1+\sum_{x\in A}\xi_x)$, for some finite set $A$. For fixed $\varphi\in C([0,T]\times\TTd)$, fixed $K$, and fixed $\delta>0$, it holds that \begin{equation}
\label{eq:replacement-statement} \limsup_{\eps\downarrow0} \limsup_N\QQ^{N,K}\left(\left|\int_0^T V^{F,N,\varphi}_\eps(t,\xi^N_t)\dd t \right|>\delta\right)=0.
\end{equation} \end{proposition} The second key step is the following $L^2_{t,x}$-estimate. \begin{proposition}[Microscopic energy estimate]\label{prop:l2-estimate}
Under the measures $\QQ^{N,K}$ defined above,
\begin{equation}\label{eq:l2-estimate-main}
{\limsup_{N\to\infty}\E_{\QQ^{N,K}}\Big[\int_0^T \|\xi_t^N\|_{L^2_N}^2\,\dd t\Big]<\infty.}
\end{equation}
\end{proposition}}

\subsection{Weak Solutions to the Truncated Fokker-Planck Equation}
We first make precise the nonlinear equation generated by the cutoff tilt. For fixed $K<\infty$, the function $\Theta_K$ defined in \eqref{eq:def-Theta-Gamma-K} is continuous on $[0,\infty)$, globally Lipschitz, and has linear growth. More precisely,
\[
\Theta_K(\rho)=\rho^2\Bigg[1-\Big(1+\frac{2K}{3\rho}\Big)e^{-K/\rho}\Bigg],
\qquad \rho>0,
\]
so that $\Theta_K(\rho)\sim \frac K3\rho$ as $\rho\to\infty$. We therefore set
\[
\Theta_K^\infty:=\lim_{\rho\to\infty}\frac{\Theta_K(\rho)}{\rho}=\frac K3.
\]
If $\xi\in \cM_+(\TTd)$ has Lebesgue decomposition $\xi=v\,dx+\xi^{\perp}$, we define the finite measure
\begin{equation}\label{eq:def-Theta-on-measures}
\Theta_K(\xi):=\Theta_K(v)\,dx+\Theta_K^\infty\xi^{\perp}.
\end{equation}
The linear growth of $\Theta_K$ ensures that this is well-defined and satisfies the growth and Lipschitz continuity
\begin{equation}\label{eq:Theta-measure-growth}
\langle 1,\Theta_K(\xi)\rangle\le C_K\langle 1,\xi\rangle;\qquad \|\Theta_K(\xi)-\Theta_K(\xi')\|_{\rm TV}\le C_K\|\xi-\xi'\|_{\rm TV}
\end{equation}
for a constant $C_K<\infty$ depending only on $K$. With this, we make precise the notion of weak, measure-valued solutions.

\begin{definition}\label{def:weak_measure_soln}
A trajectory $\xi\in C([0,T],\cM_+(\TTd))$ is a weak solution to the truncated Fokker--Planck equation
\begin{equation}\label{eq:tfp}
\partial_t v=\frac12\Delta v-\nabla\cdot\big(\chi_K\,\Theta_K(v)\nabla H\big)
\end{equation}
with initial datum $u_0\,dx$ if, for every $\varphi\in C_c^{1,2}([0,T)\times\TTd)$,
\begin{equation}\label{eq:weak-tfp}\begin{split}
\langle \varphi(0,\cdot),u_0\rangle&+
\int_0^T\Big\langle \partial_t\varphi(t,\cdot)+\frac12\Delta\varphi(t,\cdot),\xi_t\Big\rangle\,dt
\\&+\int_0^T\Big\langle \chi_K(t,\cdot)\nabla H(t,\cdot)\cdot \nabla\varphi(t,\cdot),\Theta_K(\xi_t)\Big\rangle\,dt=0.\end{split}
\end{equation}
\end{definition}

\begin{remark}\label{rmk:strong_implies_weak} If $\xi_t=v_t(x)dx$ for $v\in C^{1,2}([0,T]\times\TTd)$, it is straightforward to check that $\xi$ is a weak solution to \eqref{eq:tfp} if, and only if, $v$ is a solution to \eqref{eq:tfp} in the classical sense. 
	
\end{remark}

\begin{lemma}\label{lem:unique-tfp}
For any fixed $K<\infty$, any $u\in \cR$, and any $H\in C^\infty([0,T]\times\TTd)$, there exists at most one weak solution to \eqref{eq:tfp}. In particular, when $H$ is the control associated to $u$ by \eqref{eq:weighted-elliptic-lambda}, the trajectory $u\,dx$ is the unique weak solution to \eqref{eq:tfp}.
\end{lemma}

\begin{proof}
Set
\[
b(t,x):=\chi_K(t,x)\nabla H(t,x),
\qquad (t,x)\in [0,T]\times\TTd.
\]
Since $u$ and $H$ are smooth and $\chi_K=A_K\circ u$, the vector field $b$ is smooth and bounded.

Let $P_t=e^{\frac t2\Delta}$ denote the heat semigroup on $\TTd$. We first show that every weak solution admits the mild formulation
\begin{equation}\label{eq:mild-tfp}
\langle \phi,\xi_t\rangle
=\langle P_t\phi,u_0\rangle+
\int_0^t\big\langle b(s,\cdot)\cdot \nabla P_{t-s}\phi,\Theta_K(\xi_s)\big\rangle\,ds,
\qquad \phi\in C^2(\TTd).
\end{equation}
This follows by taking, for fixed $t\in[0,T]$, smooth test functions in \eqref{eq:weak-tfp} approximating\[
\varphi(s,x):=\mathbf 1_{[0,t]}(s) P_{t-s}\phi(x).
\]


Let $\xi, \xi'$ be two weak solutions. 
Subtracting the two mild formulations \eqref{eq:mild-tfp}, using the Lipschitz estimate \eqref{eq:Theta-measure-growth} and using the TV-smoothing estimate
\begin{equation}\label{eq:l1-smoothing}
\|\nabla P_t \xi\|_{\rm TV}\le C t^{-1/2}\|\xi\|_{\rm TV},
\qquad t\in (0,T],
\end{equation}
we obtain
\begin{align*}
\|\xi_t-\xi'_t\|_{\rm TV}
&\le \int_0^t \big\|\nabla P_{t-s}\big(b(s,\cdot)(\Theta_K(\xi_s)-\Theta_K(\xi'_s))\big)\big\|_{\rm TV}\,ds \\
&\le C_K \|b\|_\infty \int_0^t (t-s)^{-1/2}\|\xi_s-\xi'_s\|_{\rm TV}\,ds.
\end{align*}
This implies that the difference satisfies, for a new constant $C$,
\[
\|\xi_t-\xi'_t\|_{\rm TV}\le C\int_0^t (t-s)^{-1/2}\|\xi_s-\xi'_s\|_{\rm TV}\,ds.
\]
Iterating this inequality once gives
\[
\|\xi_t-\xi'_t\|_{\rm TV}\le C^2\int_0^t\int_0^s (t-s)^{-1/2}(s-r)^{-1/2}\|\xi_r-\xi'_r\|_{\rm TV}\,dr\,ds
= C^2\pi\int_0^t \|\xi_r-\xi'_r\|_{\rm TV}\,dr,
\]
which implies that $\xi=\xi'$ by Gr\"onwall's lemma.\\

Finally, for the given trajectory $u$ and associated control $H$, we recall $\chi_K\Theta_K(u)=u^2$ by the definition of $\chi_K=A_K\circ u$. By definition of $H$, $u$ is a classical solution to 
\[
\partial_t u=\frac12\Delta u-\nabla\cdot(u^2\nabla H)
\]
and the second term on the right-hand side is equal to $-\nabla\cdot(\chi_K \Theta_K(u)\nabla H)$. The given trajectory $u$ is therefore a weak solution to \eqref{eq:tfp} by Remark \ref{rmk:strong_implies_weak}, and hence the unique weak solution, in the sense of Definition \ref{def:weak_measure_soln}.
\end{proof}

\subsection{Convergence to the Specified Path}
The remainder of this subsection is dedicated to the following, which shows that each family of measures $\QQ^{N,K}$ produce a tilt under which the limit path is typical.
\begin{proposition}
\label{prop:convergence-under-Q}
Under the measures $\QQ^{N,K}$, the processes $\xi^N$ converge to $u$ in the topology of $\cX$.
\end{proposition}
We divide the proof of Proposition \ref{prop:convergence-under-Q} into several lemmata. We first identify the action of the generator against smooth test functions. To do this, we define functions $\mathfrak{D}^{N,K}_{t,\varphi}(\xi)$ by \begin{equation}\label{eq:def-local-drift-functional} \mathfrak{D}^{N,K}_{t,\varphi}(\xi):=\frac1{N^{d-2}}\sum_{x\sim y}(\varphi(y)-\varphi(x))(H(t,x)-H(t,y))\Psi_K(\xi(x),\xi(y)) \end{equation} where we define \begin{equation}\label{eq:def-psiK-early}
\Psi_K(a,b):=\int_0^1 \big(p\tau_K(b)-(1-p)\tau_K(a)\big)\big((1-p)a-pb\big)\,dp.
\end{equation} With this defined, the action of the generators takes the following form. \begin{lemma}\label{lem:gen-smooth-evaluation}
For any twice-differentiable test function $\varphi\in C^2(\TTd)$, there exist constants $\epsilon_{N,\varphi,H,K}$, vanishing as $N\to\infty$, such that, $\QQ^{N,K}$-almost surely, for all $0\le t\le T$,
\begin{equation}\label{eq:tilted-generator-smooth}
\left|\cL^{N,H,K}_t\langle \varphi, \pi_N(\xi^N_t)\rangle-\frac12\langle \Delta\varphi,\pi_N(\xi^N_t)\rangle-\mathfrak D_{t,\varphi}^{N,K}(\xi^N_t)\right|\le \epsilon_{N,\varphi,H,K}.
\end{equation} 
Moreover, the $\QQ^{N,K}$-martingale
\[
M^{N,\varphi}_t:=\langle \varphi, \pi_N(\xi^N_t)\rangle-\langle \varphi, \pi_N(\xi^N_0)\rangle-\int_0^t\cL^{N,H,K}_s\langle \varphi, \pi_N(\xi^N_s)\rangle ds
\]
has predictable quadratic variation at most
\begin{equation}
\label{eq:quad-var}
[M^{N,\varphi}]_t\le \frac{C_{\varphi,H,K}}{N^d}\int_0^t\|\xi^N_s\|_{L^2_N}^2 ds.
\end{equation}
\end{lemma}

\begin{proof}
Let
\[
F_\varphi(\xi):=\langle \varphi,\pi_N(\xi)\rangle=\frac1{N^d}\sum_{z\in\TND}\varphi(z)\xi(z).
\]
Fix an edge $\{x,y\}$ and $p\in[0,1]$. Since only the values at $x$ and $y$ are modified by the jump $\xi\mapsto \xi^{x,y,p}$, one has
\begin{align*}
F_\varphi(\xi^{x,y,p})-F_\varphi(\xi)
&=\frac{\varphi(y)-\varphi(x)}{N^d}\big((1-p)\xi(x)-p\xi(y)\big)
\end{align*}
and so
\begin{align}
L_{x,y,t}^{H,K}F_\varphi(\xi)
&=\frac{\varphi(y)-\varphi(x)}{N^d}
\int_0^1 e^{\vartheta_t^N(x,y,p,\xi)}\big((1-p)\xi(x)-p\xi(y)\big)\,dp.
\label{eq:edge-generator-linear-cutoff}
\end{align}
Because $H$ and $\chi_K$ are smooth and $\tau_K\le K$, there exists $C_{H,K}<\infty$ such that
\begin{equation}\label{eq:theta-small}
|\vartheta_t^N(x,y,p,\xi)|\le \frac{C_{H,K}}{N}
\end{equation}
uniformly in $t,\{x,y\},p\in[0,1]$ and $\xi\in \Lambda_{N,a}$. We may therefore Taylor expand the exponential to second order and write\[
e^{\vartheta_t^N(x,y,p,\xi)}=1+\vartheta_t^N(x,y,p,\xi)+R_t^N(x,y,p,\xi),
\qquad |R_t^N(x,y,p,\xi)|\le \frac{C_{H,K}}{N^2}.
\]
Substituting this into \eqref{eq:edge-generator-linear-cutoff}, we obtain
\begin{align*}
L_{x,y,t}^{H,K}F_\varphi(\xi)
&=\frac{\varphi(y)-\varphi(x)}{N^d}\int_0^1\big((1-p)\xi(x)-p\xi(y)\big)\,dp \\
&\quad +\frac{\varphi(y)-\varphi(x)}{N^d}\int_0^1 \vartheta_t^N(x,y,p,\xi)\big((1-p)\xi(x)-p\xi(y)\big)\,dp \\
&\quad +\mathcal R_{x,y,t}^N(\xi),
\end{align*}
where
\[
|\mathcal R_{x,y,t}^N(\xi)|
\le \frac{C_{\varphi,H,K}}{N^{d+3}}\big(\xi(x)+\xi(y)\big),
\]
because $|\varphi(y)-\varphi(x)|\le C_\varphi N^{-1}$. The first integral is equal to $\frac12(\xi(x)-\xi(y))$, whereas, recalling \eqref{eq:def-psiK-early}, the second equals
\[
\chi_K\!\left(t,\frac{x+y}{2}\right)\big(H(t,x)-H(t,y)\big)\Psi_K\big(\xi(x),\xi(y)\big).
\]

Summing over all unordered nearest-neighbour edges and multiplying by $N^2$ therefore yields
\begin{align}
\cL_t^{N,H,K}F_\varphi(\xi)
&=\frac{N^2}{2N^d}\sum_{x\sim y}\big(\varphi(y)-\varphi(x)\big)\big(\xi(x)-\xi(y)\big)
+\mathfrak D_{t,\varphi}^{N,K}(\xi)+\mathcal E_{N,\varphi,H,K}(\xi),
\label{eq:generator-before-discrete-lap}
\end{align}
with
\[
|\mathcal E_{N,\varphi,H,K}(\xi)|
\le \frac{C_{\varphi,H,K}}{N}\,\langle 1,\pi_N(\xi)\rangle.
\]
By discrete summation by parts,
\[
\frac{N^2}{2N^d}\sum_{x\sim y}\big(\varphi(y)-\varphi(x)\big)\big(\xi(x)-\xi(y)\big)
=\frac12\langle \Delta_N\varphi,\pi_N(\xi)\rangle,
\]
where $\Delta_N$ is the discrete Laplacian. Since $\varphi\in C^2(\TTd)$,
\[
\sup_{x\in\TND}|\Delta_N\varphi(x)-\Delta\varphi(x)|=c_{N,\varphi}\to 0.
\]
Consequently, using the mass bound, $\QQ^{N,K}$-almost surely, for all time,
\[
\left|\frac12\langle \Delta_N\varphi-\Delta\varphi,\pi_N(\xi_t)\rangle\right|
\le a c_{N,\varphi}
\]
and the claim \eqref{eq:tilted-generator-smooth} follows. To compute the quadratic variation of the martingale $M^{N,\varphi}_t$, note that the jump size of $F_\varphi$ along the edge $\{x,y\}$ is bounded by
\[
|F_\varphi(\xi^{x,y,p})-F_\varphi(\xi)|
\le \frac{C_\varphi}{N^{d+1}}\big(\xi(x)+\xi(y)\big).
\]
Under $\QQ^{N,K}$, the compensator of the jump measure is multiplied by the factor $r_t(x,y,p,\xi)$, which is uniformly bounded by $e^{C_{H,K}/N}\le 2$ for all sufficiently large $N$. Hence
\begin{align*}
[M^{N,\varphi}]_t
&\le C_{\varphi,H,K}\int_0^t\frac{N^2}{N^{2d+2}}\sum_{x\sim y}\big(\xi_s^N(x)+\xi_s^N(y)\big)^2\,ds \\
&\le \frac{C_{\varphi,H,K}}{N^d}\int_0^t\|\xi_s^N\|_{L^2_N}^2\,ds,
\end{align*}
which is exactly \eqref{eq:quad-var}.
\end{proof}

We next prove a tightness property.
\begin{lemma}\label{lem:tightness2}
Under $\QQ^{N,K}$, the processes $\pi_N(\xi^N)$ are tight in the topology of $D([0,T],\cM_+(\TTd))$, and all subsequential limits are almost surely valued in $C([0,T],\cM_+(\TTd))$.
\end{lemma}

\begin{proof}
Let $a>\langle 1,u_0\rangle$ be the constant fixed in the definition of $Y_0^N$. Since every jump preserves the total energy on the corresponding edge, the total mass is conserved along the dynamics. Hence, under $\QQ^{N,K}$,
\[
\langle 1,\pi_N(\xi_t^N)\rangle=\langle 1,\pi_N(\xi_0^N)\rangle\le a,
\qquad 0\le t\le T,
\]
almost surely, which proves compact containment. To prove tightness in time, fix $\varphi\in C^2(\TTd)$ and set
\[
X_t^{N,\varphi}:=\langle \varphi,\pi_N(\xi_t^N)\rangle.
\]
By Lemma~\ref{lem:gen-smooth-evaluation},
\[
X_t^{N,\varphi}=X_0^{N,\varphi}+B_t^{N,\varphi}+M_t^{N,\varphi},
\]
where
\[
B_t^{N,\varphi}:=\int_0^t\Big(\frac12\langle \Delta\varphi,\pi_N(\xi_s^N)\rangle+\mathfrak D_{s,\varphi}^{N,K}(\xi_s^N)+\varepsilon_{N,\varphi,H,K}(s)\Big)\,ds
\]
for some deterministic error satisfying $\sup_{0\le s\le T}|\varepsilon_{N,\varphi,H,K}(s)|\to 0$.

We first estimate the drift. Since $\varphi$, $H$ and $\chi_K$ are smooth and $\tau_K\le K$, there exists $C_{\varphi,H,K}<\infty$ such that
\[
|\mathfrak D_{t,\varphi}^{N,K}(\xi)|
\le \frac{C_{\varphi,H,K}}{N^d}\sum_{x\sim y}({\xi(x)+\xi(y)})
\le C_{\varphi,H,K}\langle 1,\pi_N(\xi)\rangle
\le C_{\varphi,H,K}a.
\]
The same bound clearly holds for $\langle \Delta\varphi,\pi_N(\xi_t^N)\rangle$, and for the error term $\epsilon_{N,\varphi,H,K}$, and together $B^{N,\varphi}_t$ is almost surely Lipschitz continuous in time. For the martingale part, Doob's inequality and \eqref{eq:quad-var} give
\[
\E_{\QQ^{N,K}}\Big[\sup_{0\le t\le T}|M_t^{N,\varphi}|^2\Big]
\le 4\E_{\QQ^{N,K}}\big[[M^{N,\varphi}]_T\big]
\le \frac{C_{\varphi,H,K}}{N^d}
\E_{\QQ^{N,K}}\Big[\int_0^T \|\xi_s^N\|_{L^2_N}^2\,ds\Big].
\]
By Proposition~\ref{prop:l2-estimate}, the right-hand side tends to $0$. Hence $M^{N,\varphi}$ converges to $0$ in $L^2(\QQ^{N,K};D([0,T]))$.
\\\\ It follows that, for each fixed $\varphi\in C^2(\TTd)$, the sequence $X^{N,\varphi}$ is tight in $D([0,T],\R)$, and the tightness follows by Mitoma's criterion. The continuity of subsequential limits follows by the vanishing of $M^{N,\varphi}$, as all other terms are already continuous in time. \end{proof}

\begin{lemma}[Identification of the drift term]\label{lem:replacement-drift}
Let $(N_j)_{j\ge 1}$ be a subsequence and suppose that, on some probability space where $\xi^N$ has the same law as under $\QQ^{N,K}$, the corresponding empirical measures satisfy
\[
\pi_{N_j}(\xi^{N_j})\longrightarrow \Xi
\qquad\text{almost surely in }D([0,T],\cM_+(\TTd)).
\]
Then, for every $\varphi\in C^{1,2}([0,T)\times \TTd)$,\begin{equation}\label{eq:replacement-drift-measure}
\int_0^T \mathfrak D_{s,\varphi(s,\cdot)}^{N_j,K}(\xi_s^{N_j})\,ds
\longrightarrow
\int_0^T \big\langle \chi_K(s,\cdot)\nabla H(s,\cdot)\cdot\nabla\varphi(s,\cdot),\Theta_K(\Xi_s)\big\rangle\,ds
\end{equation}
in probability.
\end{lemma}

\begin{proof}
 For $i\in\{1,\dots,d\}$, write $e_i^N:=N^{-1}e_i$. Recalling the definition \eqref{eq:def-psiK-early}, we rewrite the edge sum in \eqref{eq:def-local-drift-functional} by coordinate directions to find
\[
\mathfrak D_{s,\varphi(s,\cdot)}^{N,K}(\xi)
=\frac1{N^d}\sum_{i=1}^d\sum_{x\in\TND}
\alpha_{i}^{N}(s,x)\,\Psi_K\big(\xi(x),\xi(x+e_i^N)\big),
\]
where
\[
\alpha_i^N(s,x):=
N^2\chi_K\!\left(s,x+\frac{e_i^N}{2}\right)
\big(H(s,x)-H(s,x+e_i^N)\big)
\big(\varphi(s,x+e_i^N)-\varphi(s,x)\big).
\]
By smoothness of $\chi_K$, $H$ and $\varphi$,
\begin{equation}\label{eq:alpha-conv}
\sup_{(s,x)\in [0,T]\times\TND}
\left|\alpha_i^N(s,x)+\chi_K(s,x)\partial_iH(s,x)\partial_i\varphi(s,\cdot)(x)\right|\longrightarrow 0.
\end{equation}

Let
\[
\bar\Psi_K(\rho):=\E_{\nu_\rho^N\otimes \nu_\rho^N}\big[\Psi_K(X,Y)\big].
\]
A direct computation using the definitions of $m_1^K$ and $m_{1,1}^K$ yields
\[
\bar\Psi_K(\rho)=\frac13\rho m_1^K(\rho)-\frac23 m_{1,1}^K(\rho)=-\Theta_K(\rho).
\]
Fix $\varepsilon>0$. Applying Proposition~\ref{prop:replacement} to the local Lipschitz function
\[
F_i(\xi):=\Psi_K\big(\xi(0),\xi(e_i^N)\big),
\]
with the test function $-\chi_K(s,x)\partial_iH(s,x)\partial_i\varphi(s,\cdot)(x)$ and using \eqref{eq:alpha-conv}, we obtain
\begin{align}
&\int_0^T \mathfrak D_{s,\varphi(s,\cdot)}^{N,K}(\xi_s^N)\,ds 
\\ &=-\sum_{i=1}^d \frac1{N^d}\sum_{x\in\TND}\int_0^T
\chi_K(s,x)\partial_iH(s,x)\partial_i\varphi(s,\cdot)(x)\,\bar\Psi_K\big(\bar\xi^{N\varepsilon}_s(x)\big)\,ds+\mathcal{T}_{N,\eps}\nonumber
\\
&=\sum_{i=1}^d \frac1{N^d}\sum_{x\in\TND}\int_0^T
\chi_K(s,x)\partial_iH(s,x)\partial_i\varphi(s,\cdot)(x)\,\Theta_K\big(\bar\xi^{N\varepsilon}_s(x)\big)\,ds+\mathcal{T}_{N,\eps} \nonumber
\label{eq:drift-after-replacement}
\end{align} for error terms $\mathcal{T}_{N,\eps}$ which vanish in probability in the double-limit $N\to \infty, \eps\to 0$ in the sense of Proposition \ref{prop:replacement}. In the sequel we use the same notation for an error, possibly varying from line to line, with this property.\\ \\ 
For fixed $\varepsilon$, the almost sure convergence of $\pi_{N_j}(\xi^{N_j})$ to $\Xi$ implies the convergence of the local averages to the mollified density for almost all times
\[
\bar\xi_s^{N_j\varepsilon}(x)\longrightarrow (\iota_\varepsilon*\Xi_s)(x)
\]
 where $\iota_\varepsilon$ denotes the normalised indicator of the cube of radius $\varepsilon$, and we obtain\begin{equation*}\begin{split}
\int_0^T \mathfrak D_{s,\varphi(s,\cdot)}^{N_j,K}(\xi_s^{N_j})\,ds &= 
\int_0^T\!\int_{\TTd}
\chi_K(s,x)\nabla H(s,x)\cdot\nabla\varphi(s,x)
\,\Theta_K\big((\iota_\varepsilon*\Xi_s)(x)\big)\,dx\,ds \\&\hspace{3cm} + \mathcal{T}_{N,\eps}.
\end{split}\end{equation*}
Finally, thanks to Proposition \ref{prop:l2-estimate} and lower semicontinuity, $\Xi$ is almost surely valued in $L^2_{t,x}\subset \cX$. It follows that $\iota_\varepsilon*\Xi_s$ converge to $\Xi_s$ in $L^2_{x}$ and hence in $L^1_{x}$ for almost all $t$, and for such $t$ the same holds for $\Theta_K(\iota_\varepsilon*\Xi_s)\to\Theta_K(\Xi_s)$ by \eqref{eq:Theta-measure-growth}.  As $\varepsilon\downarrow 0$, the integral in the previous display thus converges almost surely, and hence in probability, to the integral appearing in \eqref{eq:replacement-drift-measure}, and the proof is complete.
\end{proof}

We are now ready to assemble the various lemmata and give the

\begin{proof}[Proof of Proposition \ref{prop:convergence-under-Q}]
By Lemma~\ref{lem:tightness2}, the sequence $\pi_N(\xi^N)$ is tight in $D([0,T],\cM_+(\TTd))$. Let $(N_j)_{j\ge 1}$ be a subsequence along which there is convergence in distribution. By Skorokhod's representation theorem, we may realise this subsequence and a limit process $\Xi$ on a common probability space in such a way that
\[
\pi_{N_j}(\xi^{N_j})\longrightarrow \Xi
\qquad\text{almost surely in }D([0,T],\cM_+(\TTd)).
\]
By Lemma~\ref{lem:svle}, the initial data converge in probability to $u_0\,dx$, hence
\begin{equation}\label{eq:init-convergence-propQ}
\Xi_0=u_0\,dx
\qquad\text{almost surely.}
\end{equation}

By the second part of Lemma \ref{lem:tightness2}, $\Xi\in C([0,T],\cM_+(\TTd))$ almost surely. Fix $\varphi\in C^{1,2}([0,T)\times \TTd)$. By Lemma~\ref{lem:gen-smooth-evaluation},
\begin{align}
0
&=\langle \varphi_0,\pi_{N_j}(\xi_0^{N_j})\rangle
+\int_0^T\left\langle \partial_s\varphi(s,\cdot)+\frac12\Delta\varphi(s,\cdot),\pi_{N_j}(\xi_s^{N_j})\right\rangle\,ds
 \nonumber\\
&\hspace{3cm}+\int_0^T \mathfrak D_{s,\varphi(s,\cdot)}^{N_j,K}(\xi_s^{N_j})\,ds+M_T^{N_j,\varphi}+\int_0^T \varepsilon_{N_j,\varphi,H,K}(s)\,ds.
\label{eq:martingale-identity-cutoff}
\end{align}
The error $\int \varepsilon$ tends to $0$ uniformly in $t$. For the martingale term, Doob's inequality, \eqref{eq:quad-var} and Proposition~\ref{prop:l2-estimate} imply
\[
\E\Big[\sup_{0\le s\le T}|M_s^{N_j,\varphi}|^2\Big]
\le \frac{C_{\varphi,H,K}}{N_j^d}
\E\Big[\int_0^T \|\xi_s^{N_j}\|_{L^2_N}^2\,ds\Big]\longrightarrow 0.
\]
Therefore
\begin{equation}\label{eq:martingale-vanishes-cutoff}
\sup_{0\le s\le T}|M_s^{N_j,\varphi}|\longrightarrow 0
\qquad\text{in probability.}
\end{equation}
By the almost sure convergence of $\pi_{N_j}(\xi^{N_j})$ to $\Xi$ and  Lemma~\ref{lem:replacement-drift}, the two integrals on the first line of \eqref{eq:martingale-identity-cutoff} converge in probability to
\[
\int_0^T \left\langle \partial_s\varphi(s,\cdot)+\frac12\Delta\varphi(s,\cdot),\Xi_s\right \rangle\,ds +\int_0^T \big\langle \chi_K(s,\cdot)\nabla H(s,\cdot)\cdot\nabla\varphi(s,\cdot),\Theta_K(\Xi_s)\big\rangle\,ds.
\]
Passing to the limit in \eqref{eq:martingale-identity-cutoff} and using \eqref{eq:init-convergence-propQ}, we conclude that, almost surely,\begin{equation*}\begin{split}
0
&=\langle \varphi,u_0\rangle
+\int_0^T\left\langle (\partial_s+\frac12\Delta)\varphi(s,\cdot),\Xi_s\right\rangle\,ds
\\&\hspace{3cm}+\int_0^t\big\langle \chi_K(s,\cdot)\nabla H(s,\cdot)\cdot\nabla\varphi(s,\cdot),\Theta_K(\Xi_s)\big\rangle\,ds.
\end{split}\end{equation*}
By a diagonal argument, a further subsequence may be chosen such that the above convergence is almost sure, rather than in probability, for every $\varphi$ belonging to a countable, dense subset $\{\varphi_\ell: \ell\in \mathbb{N}\}\subset C^{1,2}([0,T)\times\TTd)$, and it follows that the limit process $\Xi$ is almost surely a weak solution to \eqref{eq:tfp}. By Lemma~\ref{lem:unique-tfp}, the only such solution is $u\,dx$. Hence $\Xi_t=u(t,\cdot)\,dx$ for every $t\in[0,T]$ almost surely, and since convergence in $D([0,T],\cM_+(\TTd))$ implies convergence in $\cX$, Proposition \ref{prop:convergence-under-Q} is complete\end{proof}

\subsection{Convergence of the Entropy}
The goal of this subsection is to prove the following, which provides the second input to the entropy method and thus completes the proof of the restricted lower bound. \begin{lemma}\label{lem:entropy-est}[Entropy estimate for $\QQ^{N,K}$]
The probability measures $\QQ^{N,K}$ produced by \eqref{eq:RND-tilt-3} satisfy
\begin{equation}
	\begin{split}
\limsup_{N\to\infty} \frac1{N^d}{\rm Ent}(\QQ^{N,K}|\PP)&\le S_\rho(u_0)+\frac12\int_0^T\!\int_{\TTd}R_K(u(t,x))u(t,x)^2|\nabla H(t,x)|^2\,\dd x\,\dd t
\\& \le \mathcal{I}_\rho(u)+\delta.
\end{split}\end{equation}
\end{lemma}
The next lemma isolates the local functions entering, respectively, the drift and the entropy of the cutoff tilt. We set
\begin{equation}\label{eq:def-xiK}
\Upsilon_K(a,b):=\int_0^1 \big(p(b\wedge K)-(1-p)(a\wedge K)\big)^2\,\dd p.
\end{equation}

\begin{lemma}[Replacement for the entropy local function]\label{lem:replacement-entropy}
As $N\to \infty$, we have the convergence in $\QQ^{N,K}$-probability
\begin{equation}\label{eq:replacement-entropy}\begin{split} &
\frac1{N^d}\int_0^T\!N^2\sum_{x\sim y}\big(H(t,x)-H(t,y)\big)^2\chi_K\!\left(t,\frac{x+y}{2}\right)^2\Upsilon_K\big(\xi_t^N(x),\xi_t^N(y)\big)\,\dd t
\\& \hspace{5cm}\to \int_0^T\!\int_{\TTd}\chi_K(t,x)^2\Gamma_K(u(t,x))|\nabla H(t,x)|^2\,\dd x\,\dd t. \end{split}
\end{equation}
\end{lemma}
\begin{proof}
We rewrite the discrete functional in \eqref{eq:replacement-entropy} as
\[
\sum_{i=1}^d\frac1{N^d}\sum_{x\in\TND}\int_0^T \beta_i^N(t,x)\,
\Upsilon_K\big(\xi_t^N(x),\xi_t^N(x+e_i^N)\big)\,dt,
\]
where
\[
\beta_i^N(t,x):=N^2\big(H(t,x)-H(t,x+e_i^N)\big)^2\chi_K\!\left(t,x+\frac{e_i^N}{2}\right)^2.
\]
By smoothness of $H$ and $\chi_K$,
\[
\sup_{(t,x)}\left|\beta_i^N(t,x)-\chi_K(t,x)^2|\partial_i H(t,x)|^2\right|\longrightarrow 0.
\]
Applying Proposition~\ref{prop:replacement} to the local Lipschitz function
\[
F_i(\xi):=\Upsilon_K\big(\xi(0),\xi(e_i^N)\big),
\]
we may replace $F_i$ by its local-equilibrium average. A direct computation using the definition of $\Gamma_K$ gives
\[
\bar F_i(\rho)=\E_{\nu_\rho^N\otimes \nu_\rho^N}[\Upsilon_K(X,Y)]=\Gamma_K(\rho).
\]
The replacement estimate and the convergence under $\QQ^{N,K}$ from Proposition~\ref{prop:convergence-under-Q} therefore imply, by essentially the same arguments as Lemma \ref{lem:replacement-drift}, that the discrete entropy functional converges to
\[
\int_0^T\!\int_{\TTd}\chi_K(t,x)^2\Gamma_K(u(t,x))|\nabla H(t,x)|^2\,dx\,dt,
\]
which is exactly \eqref{eq:replacement-entropy}.
\end{proof}

\begin{proof}[Proof of Lemma \ref{lem:entropy-est}]
We start by decomposing the entropy into the static and dynamic cost:
\[
\frac1{N^d}\log \frac{d\QQ^{N,K}}{d\PP}=\frac1{N^d}\log Y_0^N+\frac1{N^d}\log Z_T^N.
\]
For the first term, $\QQ^N$-almost surely, we have \begin{equation}\begin{split}\label{eq:logyno}
	\frac1{N^d}\log Y^N_0 &=\frac1{N^d}\sum_{x\in \TND}\left[\frac{\xi^N_0(x)}{\rho}-\log\left(\frac{u_0(x)}{\rho}\right)-\frac{\xi^N_0(x)}{u_0(x)}\right]\\ & -\frac1{N^d}\log \nu^N_{u_0}(\langle 1, \pi^N(\xi^N_0)\rangle \le a). \end{split}
\end{equation} In the first term, since $\xi^N_0\sim \nu^N_{u_0,a}$, we may use Lemma \ref{lem:svle} and convergence against the continuous test function $\rho^{-1}-u_0^{-1}$, and the Riemann sum approximation to the integral for $-\log(u_0/\rho)$, to obtain the convergence of the first line to $S_\rho(u_0)$ in $\QQ^N$-probability. Thanks to the total mass bound in the construction of $\QQ^N$, the left-hand side is $\QQ^N$-almost surely bounded, say by $$a\sup_x|\rho^{-1}-u_0^{-1}|+\sup|\log(u_0/\rho)|$$  and bounded convergence yields the same convergence in expectation. Meanwhile, since $a>\langle 1, u_0\rangle$, Lemma \ref{lem:svle} shows that the conditioning factor in the second line of \eqref{eq:logyno} converges to $1$, and hence the contribution to \eqref{eq:logyno} vanishes in the limit. All together, we have proven that \begin{equation}\begin{split}\label{eq:conv-init-cost}&\limsup_N \mathbb{E}_{\QQ^N}\left[\frac1{N^d}\log Y^N_0\right]=S_\rho(u_0). \end{split}\end{equation}
We now turn to the dynamic cost arising from $Z_T^N$. By Lemma~\ref{lem:girsanov},
\begin{equation}\label{eq:entropy-exact}\begin{split}
&\E_{\QQ^{N,K}} \log Z_T^N
\\&=\E_{\QQ^{N,K}}\Bigg[\int_0^T\!N^2\sum_{x\sim y}\int_0^1
\Big(r_t(x,y,p,\xi_t^N)\log r_t(x,y,p,\xi_t^N)-r_t(x,y,p,\xi_t^N)+1\Big)\,\dd p\,\dd t\Bigg].\end{split}
\end{equation}
Since $u$ and $H$ are smooth and the cutoff is bounded by $K$, one has $|\vartheta_t^N(x,y,p,\xi)|\le C_{H,K}N^{-1}$ uniformly, and hence
\[
\Big|r_t(x,y,p,\xi)\log r_t(x,y,p,\xi)-r_t(x,y,p,\xi)+1-\frac12\vartheta_t^N(x,y,p,\xi)^2\Big| \le C_{H,K}N^{-3}
\]
uniformly in $t, \{x,y\},p$ and $\xi\in \Lambda_{N,a}$. The remaining error is negligible after multiplying by $N^2$ and averaging over edges. Expanding the square and integrating in $p$ gives the local function $\Upsilon_K$, so that Lemma~\ref{lem:replacement-entropy} yields
\[
\limsup_{N\to\infty}\frac1{N^d}\E_{\QQ^{N,K}}\log Z_T^N
\le \frac12\int_0^T\!\int_{\TTd}\chi_K(t,x)^2\Gamma_K(u(t,x))|\nabla H(t,x)|^2\,\dd x\,\dd t.
\]
Since $\chi_K=A_K\circ u$, the definitions of $A_K, R_K$ at \eqref{eq:def-AK-RK} produce
\[
\frac12\int_0^T\!\int_{\TTd}R_K(u(t,x))u(t,x)^2|\nabla H(t,x)|^2\,\dd x\,\dd t.
\]
Combining this with \eqref{eq:conv-init-cost} and the choice \eqref{eq:choice-of-K} of $K$ proves the claim.
\end{proof}

\subsection{Proof of Propositions \ref{prop:replacement} - \ref{prop:l2-estimate}}\label{subsec:proof-technical-inputs}
We now prove the two technical inputs stated above, starting with \ref{prop:l2-estimate}. The strategy is to close a pathwise $L^\infty_tH^{-1}_x$-estimate, exploiting the boundedness of the drift. To do this, let $G_N$ be the Green kernel for the (massive) discrete Laplacian $-\Delta_N+1$, that is,
\begin{equation}\label{eq:green-def}
(-\Delta_{N,x}+1)G_N(x,y)=N^d1(x=y).
\end{equation} We use the facts, in the sequel, that, in any dimension, $G_N(x,y)=G_N(x-y)$ is a translationally invariant function, satisfying \begin{equation}
	\label{eq: useful facts about g} \frac1{N^d}\sum_{x\in \TND} G_N(x)=1; \qquad \limsup_N \frac{G_N(0)}{N^d}<\infty
\end{equation} and that, for any $f:\TND\to \RR$, \begin{equation} \label{eq: GN pos def}
	\langle f, G_Nf\rangle_N\ge 0; \qquad G_Nf(x):=\frac1{N^d}\sum_{y\in \TND}G_N(x-y)f(y).
\end{equation}
\begin{lemma}[Quadratic Lyapunov estimate]\label{lem:quadratic-generator}
Define
\begin{equation}\label{eq:quadratic-functional}
F_N(\xi):=\frac1{N^{2d}}\sum_{x,y\in\TND} G_N(x-y)\xi(x)\xi(y)
=\langle \xi,G_N\xi\rangle_N.
\end{equation}
There exist constants $c,C<\infty$, depending only on $K,\|\nabla H\|_{L^\infty_{t,x}}$, such that for every $N$ sufficiently large, every $t\in [0,T]$,
\begin{equation}\label{eq:quadratic-generator-estimate}
\cL_t^{N,H,K}F_N(\xi)
\le CF_N(\xi)-c\|\xi\|_{L^2_N}^2.
\end{equation}
\end{lemma}

\begin{remark}
	This is the point at which the regularity of $H$, inherited from $u\in \cR$ via Lemma \ref{lem:optimal-drift-smooth}, enters the argument. For general drift fields, the constant $C$ would diverge and we would not be able to recover Proposition \ref{prop:convergence-under-Q}. 
\end{remark}

\begin{proof}
Fix an edge $\{x,y\}$ and $p\in [0,1]$. To simplify notation, write
\[
s_{xy}:=\xi(x)+\xi(y),\qquad d_{xy}:=\xi(y)-\xi(x),\qquad m:=p-\frac12,\qquad \psi:=G_N\xi.
\]
In this notation, the jump increment $\delta^{x,y,p}:=\xi^{x,y,p}-\xi$ satisfies
\[
\delta^{x,y,p}(x)=ms_{xy}+\frac12 d_{xy},
\qquad
\delta^{x,y,p}(y)=-ms_{xy}-\frac12 d_{xy},
\]
and vanishes elsewhere. Since $G_N$ is symmetric,
\[
F_N(\xi^{x,y,p})-F_N(\xi)
=2\langle \delta^{x,y,p},G_N\xi\rangle_N+\langle \delta^{x,y,p},G_N\delta^{x,y,p}\rangle_N.
\]
We analyse these two terms separately. For the linear part, because $\delta^{x,y,p}$ is supported on $\{x,y\}$,
\[
2\langle \delta^{x,y,p},G_N\xi\rangle_N
=\frac{2}{N^d}\Big(ms_{xy}+\frac12 d_{xy}\Big)\big(\psi(x)-\psi(y)\big).
\]
By \eqref{eq:theta-small}, the tilting factor $r_t$ in \eqref{eq:def-r-cutoff}, \eqref{eq:tilted-generator-2} satisfies\[
|r_t(x,y,\xi,p)-1|\le C_{H,K}N^{-1}
\]
uniformly in all variables. After multiplication by $N^2$ and integration in $p$, this gives
\begin{align}
&N^2\int_0^1 r_t(x,y,\xi,p)2\langle \delta^{x,y,p},G_N\xi\rangle_N\,dp \nonumber\\
&\qquad=\frac{N^2}{N^d}\big(\psi(x)-\psi(y)\big)d_{xy}+\widetilde R_{N,t}^{(1)}(x,y;\xi),
\label{eq:linear-part-newtilt}
\end{align}
where the remainder satisfies
\begin{equation}\label{eq:linear-remainder-newtilt}
|\widetilde R_{N,t}^{(1)}(x,y;\xi)|
\le \frac{C_{H,K}}{N^{d}}\,s_{xy}\,N|\psi(y)-\psi(x)|.
\end{equation}
Summing the principal term in \eqref{eq:linear-part-newtilt} over all edges, using discrete
summation by parts and the definition of $\psi$, we obtain
\begin{equation} \label{eq:exact-diffusion-quadratic-new}
\frac{N^2}{N^d}\sum_{x\sim y}
(\psi(x)-\psi(y))d_{xy}
=
F_N(\xi)-\|\xi\|_{L^2_N}^2.
\end{equation}

To control the remainder, we apply Cauchy--Schwarz:
\begin{align*}
\sum_{x\sim y}|\widetilde R_{N,t}^{(1)}(x,y;\xi)|
&\le C_{H,K}\Big(\frac1{N^d}\sum_{x\sim y}s_{xy}^2\Big)^{1/2}
\Big(\frac1{N^d}\sum_{x\sim y}N^2|\psi(y)-\psi(x)|^2\Big)^{1/2} \\
&\le C_{H,K}\,\|\xi\|_{L^2_N}\,F_N(\xi)^{1/2},
\end{align*}
where we used the discrete Green identity and the definition of $\psi$ to identify the second term:
\[
\frac1{2N^d}\sum_{x\sim y}N^2|\psi(y)-\psi(x)|^2+\|\psi\|_{L^2_N}^2
=\langle \psi,(-\Delta_N+1)\psi\rangle_N
=F_N(\xi).
\]
Using Peter-Paul, for all $\zeta>0$, we find a constant $C_{\zeta,H,K}$ such that
\begin{equation}\label{eq:linear-remainder-absorbed}
\sum_{x\sim y}|\widetilde R_{N,t}^{(1)}(x,y;\xi)|
\le \zeta \|\xi\|_{L^2_N}^2+C_{\zeta,H,K}F(\xi).
\end{equation}

For the quadratic part, translation invariance gives
\[
\langle\delta_{x,y,p},G_N\delta_{x,y,p}\rangle_N
=
\frac{2\Gamma_N}{N^{2d}}
\left(ms_{xy}+\frac12 d_{xy}\right)^2
; \qquad 
\Gamma_N:=G_N(0)-G_N(e_1/N).
\]
Evaluating the Green equation at the origin yields
\[
2dN^2\Gamma_N+G_N(0)=N^d, \quad \text{ and hence }\quad  0\le \Gamma_N\le \frac{N^{d-2}}{2d}.
\]

Since $e^{\vartheta_{xy}^N(t,p,\xi)}\le \frac54$ for large \(N\),

\begin{equation}\begin{split}
N^2\sum_{x\sim y}\int_0^1
e^{\vartheta^N_{xy}(t,p,\xi)}
\langle\delta_{x,y,p},G_N\delta_{x,y,p}\rangle_N\,dp 
&\le
\frac{5}{4dN^d}
\sum_{x\sim y}
\int_0^1
\left(ms_{xy}+\frac12 d_{xy}\right)^2\,dp \\
&\le
\frac{5}{12dN^d}
\sum_{x\sim y}
\bigl(\xi(x)^2+\xi(y)^2\bigr)
=
\frac56\|\xi\|_{L^2_N}^2.
\end{split}\end{equation}
Combining this with \eqref{eq:exact-diffusion-quadratic-new} and \eqref{eq:linear-remainder-absorbed}, and choosing $\zeta<\frac16$ yields \eqref{eq:quadratic-generator-estimate}.
\end{proof}

\begin{proof}[Proof of Proposition~\ref{prop:l2-estimate}]
Let $F_N, C, c$ be as in Lemma \ref{lem:quadratic-generator}, and set
\[
H_t^N:=e^{-Ct}F_N(\xi_t^N).
\]
By Dynkin's formula,
\[
H_t^N-H_0^N-\int_0^t e^{-Cs}\big((\cL_s^{N,H,K}-C)F_N\big)(\xi_s^N)\,ds
\]
is a martingale under $\QQ^{N,K}$. Taking expectations and using Lemma~\ref{lem:quadratic-generator}, we obtain
\[
\E_{\QQ^{N,K}}[H_T^N]
+c\E_{\QQ^{N,K}}\Big[\int_0^T e^{-Cs}\|\xi_s^N\|_{L^2_N}^2\,ds\Big]
\le \E_{\QQ^{N,K}}[H_0^N]+C\int_0^T e^{-Cs}\,ds.
\]
The evaluation at $T$ is nonnegative and can be dropped thanks to \eqref{eq: GN pos def}. The initial term is uniformly bounded by recalling that $\xi^N_0\sim \nu^N_{u_0,a}$ under $\QQ^{N,K}$, which implies that, for sufficiently large $N$, $$ \mathbb{E}_{\QQ^{N,K}} \xi_0(x)\xi_0(y)\le 2u_0(x)(u_0(y)+1(x=y)u_0(x)). $$ For all such $N$, \begin{equation}\begin{split}
	\EE_{\QQ^{N,K}} F(\xi^N_0) &=\frac1{N^{2d}}\sum_{x,y\in \TND} G_N(x-y)\EE_{\QQ^{N,K}}\xi^N(x)\xi^N(y) \\ & \le 2\|u_0\|_{L^\infty_x}^2\left(\frac1{N^d}\sum_{x\in \TND} G_N(x)\right)+\frac{G_N(0)}{N^d}\|u_0\|_{L^\infty_x}^2
\end{split} \end{equation} which is bounded $\le C\|u_0\|_{L^\infty_x}^2$ as $N\to \infty$ thanks to \eqref{eq: useful facts about g}. We conclude that, for a new constant $C'$,
\[
c\E_{\QQ^{N,K}}\Big[\int_0^T \|\xi_s^N\|_{L^2_N}^2\,ds\Big]
\le C'
\]
and the proof of \eqref{eq:l2-estimate-main} is complete.
\end{proof}

Before turning to the proof of Proposition~\ref{prop:replacement}, we give the statement of superexponential estimate for {\em bounded, Lipschitz} observables.  \begin{lemma} \label{lem:supex} Let $F$ be a bounded, Lipschitz local function $F:\Lambda_N\to \RR$. Then, for every test function $\varphi\in C([0,T]\times \TTd)$ and  every $\delta>0$, the errors $V^{F,N,\varphi}_\eps(t,\xi)$ defined in \eqref{eq: def V} satisfy
\begin{equation}\label{eq:bounded-superexp-replacement}
\limsup_{\eps\downarrow0}\limsup_{N\to\infty}
\frac1{N^d}
\log \PP\left(
\left|
\int_0^T V^{F,N,\varphi}_\eps(t,\xi^N_t)\,\dd t
\right|>\delta
\right)
=-\infty .
\end{equation}

\end{lemma}

The method is essentially the same as in \cite[Theorem 2]{benois1995large}, \cite[Lemma 1.10]{kipnis1998scaling}, and so we will not reproduce it. We remark that the boundedness of $F$ is essential in order to complete the compactness arguments necessary for both the one-block estimate \cite[Lemma 4.2]{benois1995large},\cite[Lemma 3.1]{kipnis1998scaling} and the two-block estimate \cite[Lemma 4.3]{benois1995large}. We are now ready to give the \begin{proof}[Proof of Proposition~\ref{prop:replacement}]
Let $F$ be a local, Lipschitz function, and $\varphi \in C([0,T]\times\TTd)$ as in the statement of the proposition, and fix $\delta, \zeta>0$. The proof proceeds by truncation, using Proposition \ref{prop:l2-estimate}, and then transferring Lemma \ref{lem:supex}, applied to truncated functions $F^{(L)}$, to $\QQ^{N,K}$. \\ \\ We first truncate at large values. By hypothesis, there exists a finite set $A$ and a constant $C$ such that $F$ depends only on the values of $\xi$ on $A$, and
\begin{equation}\label{eq:linear-growth-local-F}
    |F(\xi)|\le C\Big(1+\sum_{z\in A}\xi(z)\Big).
\end{equation} For $L>0$, we define a local function $F^{(L)}$ by
\[
    F^{(L)}(\xi):=F(\xi\land L)
\] with the minimum understood coordinatewise, so $(\xi\land L)(x):=\xi(x)\land L$. Since $F$ is locally Lipschitz in $\xi|_A\in [0,\infty)^A$, $F^{(L)}$ is bounded and Lipschitz, and agrees with $F$ unless at least one coordinate exceeds $L$. Thanks to \eqref{eq:linear-growth-local-F}, possibly increasing the constant, we therefore obtain
\begin{equation} \label{eq: truncate F error}
    |F(\xi)-F^{(L)}(\xi)|
    \le
    \frac{C}{L}
    \Big(1+\sum_{z\in A}\xi(z)\Big)^2 .
\end{equation}

Letting $\bar{F}^{(L)}$ be the associated function of the local average, we have the decomposition \begin{equation} \begin{split} \label{eq: decompose replacement trunctation} |V^{F,N,\varphi}_\eps(t,\xi)| \le |V^{F^L,N,\varphi}_\eps(t,\xi)|&+\frac{\|\varphi\|_\infty}{N^d}\sum_{x\in\TND}\int_0^T
\big|\tau_xF(\xi)-\tau_xF^{(L)}(\xi)\big|  \\ & +\frac1{N^d}\sum_{x\in\TND}\int_0^T
\big|\bar{F}(\bar\xi^{N\eps}(x))-\bar F^{(L)}(\bar\xi^{N\eps}(x))\big|.
\end{split}\end{equation} The first term is controlled by \eqref{eq: truncate F error} and the second moment bound:
\begin{align}
&\E_{\QQ^{N,K}}\left[
\frac1{N^d}\sum_{x\in\TND}\int_0^T
\big|\tau_xF(\xi_t^N)-\tau_xF^{(L)}(\xi_t^N)\big|
\,\dd t
\right] \nonumber \\
&\qquad\le
\frac{C_F}{L}
\E_{\QQ^{N,K}}\left[
\int_0^T
\frac1{N^d}\sum_{x\in\TND}
\Big(1+\sum_{z\in A}\xi_t^N(x+z)\Big)^2
\,\dd t
\right] \nonumber \\
&\qquad\le
\frac{C_{F,A}}{L}
\E_{\QQ^{N,K}}\left[
\int_0^T
\big(1+\|\xi_t^N\|_{L^2_N}^2\big)\,\dd t
\right].
\label{eq:replacement-tail-microscopic}
\end{align}
By Proposition~\ref{prop:l2-estimate}, the right-hand side is bounded by
$C_{F,A,K}/L$, uniformly in $N$. For the fixed $\delta, \zeta>0$, we may thus choose $L$ large enough that \begin{equation} \label{eq: cutoff 1}
	\QQ^{N,K}\left(\frac1{N^d}\sum_{x\in\TND}\int_0^T
\big|\tau_xF(\xi_t^N)-\tau_xF^{(L)}(\xi_t^N)\big|
\,\dd t> \frac{\delta}{3\|\varphi\|_\infty} \right) \le \frac{\zeta}3.
\end{equation} For the functions $\bar{F}, \bar{F}^{(L)}$, taking $\nu_\rho$-expectations of \eqref{eq: truncate F error} gives
\begin{equation}\label{eq:barF-tail-bound}
    |\bar F(\rho)-\bar F^{(L)}(\rho)|
    \le
    \frac{C_{F,A}}{L}(1+\rho^2),
    \qquad \rho\ge 0 .\end{equation}
Using Jensen's inequality for the block average, and enlarging $L$ if necessary, the same argument produces, for sufficiently large $L$,\begin{equation} 	\QQ^{N,K}\left(\frac1{N^d}\sum_{x\in\TND}\int_0^T
\big|\bar{F}(\bar\xi_t^{N\eps}(x))-\bar F^{(L)}(\bar\xi_t^{N\eps}(x))\big|
\,\dd t> \frac{\delta}{3\|\varphi\|_\infty} \right) \le \frac{\zeta}3. \label{eq: cutoff 2}
\end{equation}

Choosing $L$ large enough that both \eqref{eq: cutoff 1} - \eqref{eq: cutoff 2} hold, we now deal with the first term in \eqref{eq: decompose replacement trunctation}, using Lemmata \ref{lem:ent-crude} and \ref{lem:supex}. Write $A_{N,L,\eps,\delta}$ for the event in \eqref{eq:bounded-superexp-replacement} with $F^L$ replacing $F$, $\frac\delta3$ replacing $\delta$: $$A_{N,L,\eps,
\delta}:=\left\{\left|\int_0^T V^{F^L,N,\varphi}_\eps(t,\xi^N(t)) dt\right|>\frac{\delta}{3}\right\}. $$ The entropy inequality produces
\[
    \QQ^{N,K}(A_{N,L,\eps,\delta})
    \le
    \frac{\Ent(\QQ^{N,K}\mid\PP)+\log 2}
         {\log(1+\PP(A_{N,L,\eps,\delta})^{-1})}.
\]
Thanks to Lemma \ref{lem:ent-crude}, the numerator is at most $C_{H,K}N^d$ for sufficiently large $N$ and some constant $C_{H,K}$ independent of $\eps$. Meanwhile, Lemma \ref{lem:supex} applies to $F^L$, and so we may choose $\eps_0>0$, depending on $L$, such that the denominator is at least $3C_{H,K}N^d/\zeta$ for all $0<\eps<\eps_0$ and $N$ sufficiently large, depending on $\eps$. With these choices, we have, for all $0<\eps<\eps_0$,\begin{equation}\label{eq:bounded-replacement-under-Q}
\limsup_{N\to\infty}
\QQ^{N,K}\left(A_{N,L,\eps,\delta}\right)\le \frac\zeta3.
\end{equation} Gathering \eqref{eq: cutoff 1}, \eqref{eq: cutoff 2}, \eqref{eq:bounded-replacement-under-Q} and returning to \eqref{eq: decompose replacement trunctation}, it follows that, for all sufficiently small $\eps>0$, \begin{equation}
	\limsup_{N\to \infty}\QQ^{N,K}\left(\left|\int_0^T V^{F,N,\varphi}_\eps(t,\xi^N_t) dt\right|\ge \delta \right) \le \zeta
\end{equation} and, since $\delta, \zeta>0$ were arbitrary, the proof of Proposition \ref{prop:replacement} is complete. 
\end{proof}
\section{Construction of Pathological Trajectories} \label{sec:pathological}

In this section we prove the three cases of Theorem \ref{thm:intro-pathological}. In light of Theorem \ref{thm:intro-restricted-lb}, it is sufficient to argue at the level of the lower semicontinuous envelope of the set $\cR$. For each part of the theorem, we construct $u^n\in \cR$ converging to the desired $\Xi$, and with \begin{equation} \limsup_n \cJ(u^n)\begin{cases} <\infty, & d=1,2;\\ =0, &d\ge 3.\end{cases} \end{equation} In each case, we will construct $u^n$ with a lower bound $u^n_0\ge c>0$, and the statements in the theorem follow using the bound $$S_\rho(u^n_0)\le \frac{\langle 1, u^n_0\rangle}{\rho} + (\log c/\rho)_-.$$
It is convenient to introduce some machinery common to all three cases. In any dimension, let $\rho:\RRd\to [0,\infty)$ be a smooth, radial function, with $\int \rho dx=1$ and $\text{supp}(\rho)\subset \{|x|< \frac12\}$, which we further take to satisfy, for every $\theta>0$ \begin{equation}\label{eq: power law integrability}\int_{\RRd}\frac{|\nabla \rho(x)|^2}{\rho(x)^{2-\theta}}dx<\infty. \end{equation} For $\sigma\in(0,1]$, define
\begin{equation}\label{eq:rho-sigma}
\rho_\sigma(x):=\frac1{\sigma^d}\rho\left(\frac{x}{\sigma}\right)\end{equation} which we now understand to be a function on the torus $\TTd$, thanks to the support bound and the restriction on $\sigma$. We also fix a smooth, nondecreasing temporal cutoff function $\psi$ satisfying \begin{equation}
	\label{eq:psi construct} \begin{cases} \psi(t)=1, & t\le \frac32; \\ \psi(t)=t, & t\ge 2. \end{cases}
\end{equation}

We first give a general lemma. \begin{lemma}
	\label{lem:rho-disipation-est} Let $\rho$ be a cutoff satisfying \eqref{eq: power law integrability}. Then, for any $\theta\in(0,1)$, there exists a constant $C<\infty$ such that, for all $\epsilon>0, \sigma\in(0,\frac12)$, it holds that \begin{equation}
		\int_{\TTd} \frac{\epsilon \rho_\sigma(x)}{1+\epsilon \rho_\sigma(x)}|\nabla \log \rho_\sigma(x)|^2 dx \le C\epsilon^{1-\theta}\sigma^{d\theta-2}.
	\end{equation}
\end{lemma}

\begin{proof} Write $I(\epsilon,\sigma)$ for the integral on the left-hand side. Using the definition of $\rho_\sigma$ and changing variables, we obtain \begin{equation}
	\begin{split} I(\epsilon,\sigma)=\sigma^{d-2}\int_{\RRd} \frac{\epsilon \rho(x)}{\sigma^d+\epsilon \rho(x)}|\nabla \log \rho(x)|^2 dx.\end{split}
\end{equation} Using the inequality $\frac{z}{1+z}\le z^{1-\theta}$ for $z\ge 0$, the prefactor is at most $$ \frac{\epsilon \rho(x)}{\sigma^d+\epsilon \rho(x)}\le \epsilon^{1-\theta} \sigma^{-d(1-\theta)}\rho(x)^{1-\theta}$$ and substituting this bound into the previous display produces \begin{equation}
	I(\epsilon,\sigma)\le \epsilon^{1-\theta}\sigma^{d\theta-2}\int_{\RRd}\frac{|\nabla \rho(x)|^2}{\rho(x)^{1+\theta}}dx \le C_\theta \epsilon^{1-\theta}\sigma^{d\theta-2}
\end{equation} where $C_\theta$ is the remaining integral, which is finite thanks to the hypothesis \eqref{eq: power law integrability}.
	
\end{proof}

\subsection{Dimension $d=1$: Appearance of Singularities}

We now prove \begin{proposition}
	Let $t_0\in [0,T]$. Then there exist a sequence of paths $u^n\in \cR, u^n_0\ge 1$ converging to $\Xi\in C([0,T],\cM_+(\TT))$ such that the unique continuous representative $\Xi(dt,dx)=\xi_t(dx)dt$ is singular with respect to Lebesgue measure at time $t_0$, and so that \begin{equation}
		\limsup_n \cJ(u^n)<\infty.
	\end{equation}
\end{proposition} \begin{proof}
	Fix $t_0\in [0,T], x_0\in \TT$, and, for $\rho_\sigma, \psi$ given by \eqref{eq:rho-sigma}, \eqref{eq:psi construct}, define \begin{equation}
		\label{eq: def un d1} u^n(t,x):=1+\rho_{r_{n,t}}(x-x_0); \qquad r_{n,t}:=\sigma_0 n^{-1}\psi(n|t-t_0|^{1/2})
	\end{equation} where $\sigma_0<1/(2\sqrt{T})$, ensuring that $r_{n,t}<\frac12$ for all $n,t$. Thanks to the smoothness of $\rho, \psi$, each $u^n\in \cR$. Defining further \begin{equation}
		\xi_t(dx)=\begin{cases}
			(1+\rho_{\sigma_0|t-t_0|^{1/2}}(x-x_0))dx; & t\neq t_0; \\ dx+\delta_{x_0}(dx), &t=t_0
		\end{cases}
	\end{equation} it holds that $\xi\in C([0,T],\cM_+(\TTd))$, and that $\xi_{t_0}$ is singular with respect to Lebesgue measure. To see the convergence $u^n\to \xi$, we note that $t\neq t_0$, $u^n(t,x)dx=\xi_t(dx)$ as soon as $n>2|t-t_0|^{-1/2}$, which implies $$d(u^n,\xi)\le \frac8{n^2}.$$ It remains to estimate the dynamic cost. By direct computation, $u^n$ satisfies the skeleton equation \eqref{eq: sk} for the control \begin{equation}
		g^n(t,x):=\frac{\nabla \rho_{r_{n,t}}(x-x_0)}{2(1+\rho_{r_{n,t}}(x-x_0))}+\frac{\rho_{r_{n,t}}}{1+\rho_{r_{n,t}}}\frac{r'_{n,t}}{r_{n,t}}(x-x_0)1(|x-x_0|\le r_{n,t})
	\end{equation} understood as a function of $x\in \TT$ by covering the support of $\rho_{r_{n,t}}(x-x_0)$ by a single coordinate chart. For fixed time, the first term is smaller than the $L^2_x$-norm estimated by Lemma \ref{lem:rho-disipation-est} as \begin{equation}
		\left\|\frac{\nabla \rho_{r_{n,t}}(\cdot-x_0)}{1+\rho_{r_{n,t}}(\cdot-x_0)}\right\|_{L^2_x}^2 \le C_\theta r_{n,t}^{\theta-2}
	\end{equation} for any $\theta>0$, while the squared $L^2_x$-norm of the second term is $\sim (r'_{n,t})^2r_{n,t}$, and all together \begin{equation}
		\int_{\TT}|g^n(t,x)|^2dx \le C(r_{n,t}^{\theta-2}+(r'_{n,t})^2r_{n,t}).
	\end{equation} The first term is integrable in $t$, uniformly in $n$, for any $\theta>0$. In the second term, we have the uniform-in-$n$ comparison \begin{equation}
		|r'_{n,t}|\le C|t-t_0|^{-1/2}; \qquad |r_{n,t}|1(r'_{n,t}\neq 0)\le C|t-t_0|^{1/2}
	\end{equation} which implies that $(r_{n,t}')^2r_{n,t} \le C|t-t_0|^{-1/2}$ is also integrable. All together, we find that \begin{equation}
		\limsup_n \cJ(u^n)\le \limsup_n \|g^n\|_{L^2_{t,x}}^2 <\infty
	\end{equation} and the proof is complete.
\end{proof}
\subsection{Dimension $d=2$: Countably Many Jumps}

In this subsection we prove 

\begin{proposition}
	\label{thm: prescribed jumps d2}
Let \((t_k)_{k\ge1}\subset(0,T)\) be a sequence of distinct times and let $\epsilon_k, k\ge 0$ satisfy
\begin{equation} \label{eq: jump summability}
        \sum_{k=1}^\infty\epsilon_k^{\gamma}<\infty \qquad \text{ for some }0<\gamma<1.
\end{equation}
Then there exist $u^n\in \cR$, with $\sup_{n\ge1}\mathcal J(u^n)<\infty$ and satisfying $u^n_0\ge 1$, converging to a limit $\Xi\in \cX$ whose unique c{\`a}dl{\`a}g representative $\xi$ satisfies
\[
        \|\xi_{t_k}-\xi_{t_k-}\|_{\rm TV}=2\varepsilon_k,
        \qquad k\ge1 .
\]
\end{proposition} 

\begin{proof} We proceed by steps. In a first step, we construct a basic building block of a rapidly translated spike, with a modified cost estimate. In subsequent steps we construct a sequence of trajectories $\bar{u}^n$, each with only finitely many jumps and whose limit has the desired properties, and then approximate each $\bar{u}^n$ with smooth trajectories. 

\step{1. Definition of a Rapidly Translated Spike} We first define a three-parameter family of trajectories. As well as the functions $\rho(x,\sigma)=\rho_\sigma(x)$ introduced in \eqref{eq:rho-sigma}, let $\psi$ be as given in \eqref{eq:psi construct}, and for $a,b\in \TT^2$, let $\varphi(t,a,b)$ be a smooth path, starting at $\varphi(0,a,b)=a$ and ending at $b$, with all derivatives vanishing at the endpoints $t=0,1$. We now set\begin{equation}
	\label{eq: def moving bump} v_t(x; \epsilon,t_0,\tau,\lambda,a,b):=\epsilon \rho_{r_t}(x-s_t) \end{equation} where the radius $r_t$ and centre $s_t$ of the bump are given by \begin{equation}\label{eq: rt} r_t:=\begin{cases}
		\sigma_0\tau \psi((t_0-\lambda-t)^{1/2})/\tau), & t< t_0-\lambda; \\ \sigma_0\tau, & t_0-\lambda\le t< t_0; \\ \sigma_0\tau\psi((t-t_0)^{1/2}/\tau), & t\ge t_0
	\end{cases}
\end{equation} where $\sigma_0<1/(2\sqrt{T})$, and \begin{equation} s_t:=\begin{cases} a, & t< t_0-\lambda; \\ \varphi((t-t_0+\lambda)/\lambda, a, b), & t_0-\lambda\le t< t_0; \\ b & t\ge t_0.\end{cases}\end{equation} We write also $\bar{v}_t=\bar{v}_t(dx;\epsilon, t_0,a,b)$ for the discontinuous but c{\`a}dl{\`a}g path of measures obtained in the limit, relative to the topology of $\cX$, as $\tau,\lambda\downarrow 0$: \begin{equation}\label{eq: barv}
	\bar{v}_t(dx):=\begin{cases} \epsilon \rho_{\sigma_0(t_0-t)^{1/2}}(x-a)dx, & t<t_0; \\ \epsilon \delta_b(dx), & t=t_0\\  \epsilon \rho_{\sigma_0(t-t_0)^{1/2}}(x-b)dx, &t> t_0. \end{cases}
\end{equation} We note that each such $\bar{v}_t$ is c{\`a}dl{\`a}g, and continuous except at $t_0$, where $\bar{v}_{t_0-}=\epsilon \delta_a$. \\\\ \step{2. Inhomogeneous Cost Estimate} By construction, each such $v_t(x)$ is smooth, and we now exhibit a vector field $h$ such that $v$ solves the skeleton equation \eqref{eq: sk} for $h$, with an estimate on the weighted $L^2_{t,x}$-norm, for any $c> 0$ and $\theta\in (0,1)$, \begin{equation}
	\label{eq: inhomogenous cost} \int_0^T\int_{\TT^2} \frac{v}{v+c}|h(t,x)|^2 dx dt \le C\left[\epsilon^{1-\theta}\left(1+\frac{\lambda}{\tau^{2-2\theta}}\right) +\epsilon(1+|\log \epsilon|)+\frac{\tau^2}{\lambda}\right]
\end{equation}	for some constant $C$ depending on $c, \theta$ and the functions $\psi, \varphi$ and their derivatives, but not on $\epsilon, \tau,\sigma$ or the coordinates $t_0, a, b$. For the remainder of this step, we write $C$ for such a constant, which is allowed to change from line to line. First, we observe that $v_t$ satisfies the advection equation \begin{equation}
	\label{eq: adv} \partial_tv = -\nabla\cdot(v_t b_t)
\end{equation} for the drift field given by \begin{equation}
	b_t(x):=\left[\frac{r'_t}{r_t}(x-s_t)+s'_t\right]1(|x-s_t|\le r_t)
\end{equation} understood as a function of $x\in \TT^2$ by covering the support of $v_t$ by a single coordinate chart. We remark that, although $b_t$ is not continuous at the boundary of its support, the product $v_tb_t$ is smooth, thanks to the support condition on $\rho$.  We define the candidate $h$ on the interior of the support of $v$ by \begin{equation}
	\label{eq: def candidate h} h(t,x):=\frac12\nabla \log v_t(x)+b_t(x)
\end{equation} extending $h$ by zero outside of this support. We therefore estimate the norm appearing in \eqref{eq: inhomogenous cost} by \begin{equation}\begin{split}
	\int_0^T \int_{\TT^2}\frac{v(t,x)}{v(t,x)+c}|h(t,x)|^2 dx dt &\le \frac12 \int_0^T \int_{\TT^2}\frac{v(t,x)}{v(t,x)+c}|\nabla \log v(t,x)|^2 dx dt \\[0.5ex] & +2\int_0^T \int_{\TT^2}\frac{v(t,x)}{v(t,x)+c}|b(t,x)|^2 dx dt. \end{split}
\end{equation} The first term is estimated by Lemma \ref{lem:rho-disipation-est}, and produces, for any $\theta\in(0,1)$, \begin{equation}
	\int_{\TT^2} \frac{v(t,x)}{v(t,x)+c}|\nabla \log v(t,x)|^2 dx  \le C\epsilon^{1-\theta}r_t^{2\theta-2}.\end{equation} Integrating in time, and accounting for the definitions of $r_t$ on the different intervals gives \begin{equation} \begin{split}\label{eq: first term of inhom cost}  \int_0^T \int_{\TT^2}\frac{v(t,x)}{c+v(t,x)}|\nabla \log v(t,x)|^2 dx dt &\le {C\epsilon^{1-\theta}}\left(\lambda \tau^{2\theta-2}+2\int_0^Tt^{\theta-1}dt\right) \\& \le {C\epsilon^{1-\theta}}\left(1+\lambda\tau^{2\theta-2}\right).\end{split}\end{equation} Meanwhile, in the second term, squaring and integrating over the domain $|x-s_t|\le r_t$ produces \begin{equation} \label{eq: integration of b}
\begin{split}
	\int_{\TT^2}\frac{v}{v+c}|b_t|^2dx
	&=
	r_t^2
	\int_{B(0,1)}
	\frac{\epsilon r_t^{-2}\rho(y)}
	     {\epsilon r_t^{-2}\rho(y)+c}
	|r_t'y+s_t'|^2dy                                      \\
	&\le C\left[
	(r_t')^2 r_t^2
	\int_{B(0,1)}
	\frac{\epsilon \rho(y)}
	     {\epsilon \rho(y)+c r_t^2}
	|y|^2dy
	+
	|s_t'|^2 r_t^2
	\int_{B(0,1)}
	\frac{\epsilon \rho(y)}
	     {\epsilon \rho(y)+c r_t^2}
	dy \right].
\end{split}
\end{equation} Splitting the integral, depending on whether or not $\epsilon \rho(y)/(cr_t^2)>1 $, we obtain \begin{equation}
\label{eq:weighted-drift-pointwise}
	\int_{\TT^2}\frac{v}{v+c}|b_t|^2dx
	\le 
	C|s_t'|^2\min\left\{r_t^2,\frac{\epsilon}{c}\right\}+ C(r_t')^2\min\left\{r_t^2,\frac{\epsilon}{c}\right\}.
\end{equation}We now integrate in time and estimate the two contributions separately. The first term vanishes except in the interval $[t_0-\lambda,t_0]$, while on this interval
$r_t=\tau$ and $|s_t'|\le C \lambda^{-1}$, giving the contribution
\begin{equation}
\begin{split}
	\label{eq:translation-drift-cost}
	\int_0^T |s_t'|^2\min\left\{r_t^2,\frac\epsilon c\right\}\,dt
		&\le 
	\frac{C\tau^2}\lambda.
\end{split}
\end{equation}
Meanwhile, the second term of \eqref{eq:weighted-drift-pointwise} vanishes on
$[t_0-\lambda,t_0]$ and only the two outer intervals matter. We deal with $t\le t_0-\lambda$; the interval $t\ge t_0$ is identical. Using the definition \eqref{eq: rt}, we have the comparisons \begin{equation}
	\label{eq: comparisons for rt} |r'_t|\le C(t_0-\lambda-t)^{-1/2}; \qquad |r_t|1(r'_t\neq 0)\le C(t_0-\lambda-t)^{1/2}
\end{equation} for some $C$ depending on $\psi$. Substituting these bounds into the term of \eqref{eq:weighted-drift-pointwise} involving $r'_t$ and splitting the integral, the contribution from the time interval considered is at most \begin{equation} \label{eq: shrink cost}
	\begin{split}   \int_0^{t_0-\lambda} (r'_t)^2\min\left\{r_t^2, \frac \epsilon c \right\} dt  &\le C\int_0^{t_0}u^{-1}\min\left\{u, \frac{\epsilon}{c}\right\} du  \\ & \hspace{1cm}\le C\epsilon(1+|\log \epsilon|). \end{split} 
\end{equation}  The contribution to  $\int_{t,x} \frac{v}{v+c}|b|^2$ from the time interval $[t_0,T]$ is identical. Gathering (\ref{eq:weighted-drift-pointwise}, \ref{eq:translation-drift-cost}, \ref{eq: shrink cost}), we conclude that
\begin{equation}
	\int_0^T\int_{\TT^2}
		\frac{v}{v+c}|b_t|^2\,dxdt
	\le C\left[
	\epsilon(1+|\log \epsilon|)
	+
	\frac{\tau^2}{\lambda}\right].
\end{equation} Together with \eqref{eq: first term of inhom cost}, we have proven the claim \eqref{eq: inhomogenous cost}. \\\\ \step{3. Construction of Jump Trajectory \& Smooth Approximations.} \\ We are  now ready to give the construction of a path $u$ with countably many jumps, as well as smooth approximations. For $\gamma$ as in the hypothesis \eqref{eq: jump summability}, choose $0<\theta\le 1-\gamma$, so that $1-\theta\ge \gamma$. Recalling the notation $\bar{v}_t$ from step 1, we choose any sequences of distinct points $a_k\neq b_k$, and set \begin{equation}
	\xi_t(dx):=dx+\sum_{k\ge 1}\bar{v}_t(dx;\epsilon_k,t_k,a_k, b_k)
\end{equation} as well as \begin{equation}
	\bar{u}^n_t(dx):=dx+\sum_{k=1}^n \bar{v}_t(dx;\epsilon_k,t_k,a_k,b_k);
\end{equation} \begin{equation} \label{eq: construct smooth paths}
	{u}^{n,m}_t(x):=1+\sum_{k=1}^n {v}_t(x;\epsilon_k,t_k,\tau_{k,m}, \lambda_{k,m},a_k,b_k)
\end{equation} taking the parameters to be\begin{equation}
	\label{eq: clever param choices} \tau_{k,m}:=\frac{1}{m}; \qquad \lambda_{k,m}:=\frac{\epsilon_k^{(\theta-1)/2}}{m^{2-\theta}}.
\end{equation} To shorten notation, we write $v^{k,m}$ for the summand in \eqref{eq: construct smooth paths}, and let $\Xi(dt,dx)=\xi_t(dx)dt$. Each $u^{n,m}\in \cR$ has the desired regularity, and from the remark at the end of step 1, $$d(u^{n,m},\bar{u}^n)\to 0$$ as $m\to \infty$ for $n$ fixed, while the summability of $\epsilon_k$ ensures that $u$ is well-defined, and that $$\sup_t W(\bar{u}^n_t,\xi_t)\to 0,$$ which further implies that $d(\bar{u}^n, \Xi)\to 0$. Each $\bar{u}^n$ is c{\`a}dl{\`a}g, and hence so is $\xi$, by uniform convergence. For each $k$ and each $n\ge k$, every term with $r\neq k$ in the sum defining $\bar{u}^n$ is continuous at $t_k\neq t_r$, so $ \bar{u}^n_{t_k}-\bar{u}^n_{t_k-}=\epsilon_k(\delta_{b_k}-\delta_{a_k})$, and using the uniformity of convergence again to send $n\to \infty$, the same is true for $\xi_{t_k}-\xi_{t_k-}$. \\\\ It remains to show that $\limsup_n \limsup_m \cJ(u^{n,m})<\infty$ remains finite in the double-limit where $m\to \infty$ first with $n$ fixed, and then $n\to \infty$. With this estimate in hand, the claim follows by taking $u^n:=u^{n,m_n}$ for some suitably fast-growing $m_n\to \infty$. For the estimate of the dynamic cost, we observe that a (not necessarily optimal) $g$ for which $u^{n,m}$ solves \eqref{eq: sk} is given by \begin{equation}
	g^{n,m}:=\frac{\sum_{k=1}^n v^{k,m}h^{k,m}}{1+\sum_{k=1}^n v^{k,m}}
\end{equation} where $h^{k,m}$ are the controls given in Step 2 for each moving bump $v^{k,m}$, with the parameter choices \eqref{eq: clever param choices}. Since each $v^{k,m}\ge 0$, convexity gives the pointwise inequality \begin{equation}
	|g^{n,m}|^2\le \sum_{k=1}^n \frac{v^{k,m}}{1+\sum_{p=1}^n v^{p,m}}|h^{k,m}|^2 \le \sum_{k=1}^n\frac{v^{k,m}}{1+v^{k,m}}|h^{k,m}|^2.
\end{equation}The space-time integral of each term on the right-hand side is estimated by \eqref{eq: inhomogenous cost}, and substituting the parameter choices \eqref{eq: clever param choices} produces \begin{equation}\begin{split} \label{eq: cost asymptotic}
	&\int_0^T\int_{\TT^2} \frac{v^{k,m}}{1+v^{k,m}}|h^{k,m}|^2 dt dx \\& \hspace{1cm}\le C\left[\epsilon_k^{1-\theta}(1+\epsilon_k^{(\theta-1)/2}m^{-\theta})+\epsilon_k(1+|\log \epsilon_k|)+\epsilon_k^{(1-\theta)/2}m^{-\theta}\right] \\ & \hspace{1cm}\le C(\epsilon_k^{1-\theta}+\epsilon_k(1+|\log \epsilon_k|))+Cm^{-\theta}\epsilon_k^{(1-\theta)/2}.\end{split}
\end{equation} Since $1-\theta\ge \gamma$ and $\sum \epsilon_k^\gamma<\infty$ by hypothesis, the first two terms are summable in $k$, uniformly in $n,m$, and we find \begin{equation}
	\cJ(u^{n,m})\le \|g^{n,m}\|_{L^2_{t,x}}^2 \le C\sum_{k=1}^\infty(\epsilon_k^{1-\theta}+\epsilon_k(1+|\log \epsilon_k|))+Cm^{-\theta}\sum_{k=1}^n \epsilon_k^{(1-\theta)/2}.
\end{equation} For each fixed $n$, we may thus choose $m=m_n$ large enough that the contribution of the last term is at most $n^{-1}$ to conclude that \begin{equation}
	\limsup_{n\to\infty}\limsup_{m\to \infty}\cJ(u^{n,m})<\infty
\end{equation} and the proof is complete.\end{proof}

\subsection{Dimension $d\ge3$: Arbitrary Trajectories with Vanishing Dynamical Cost}

In this section we prove \begin{proposition}\label{thm:dge3-arbitrary-relaxed-profiles}
Let $d\ge3$ and fix $c,M\in(0,\infty)$. Let $\Xi\in \cX_M$, and suppose that
\[
\Xi(\dd t,\dd x)\ge c\,\dd t\,\dd x.
\] Then there exists a sequence $u^n\in \cR, u^n\to \Xi$, such that \begin{equation}
	\label{eq: J is awful in d>=3} \limsup_n \cJ(u^n)=0;\qquad u^n_0(x)\ge c.
\end{equation}
\end{proposition} 
\begin{proof}

Fix $\Xi$. We first define a four-parameter family $\Xi^{\tau,L,\sigma,\tau'}=u^{\tau,L,\sigma,\tau'}dtdx$ through successive approximation steps, and then show that we recover convergence to $\Xi$ in any scaling regime $\tau\to 0, \tau'\to 0, L\to \infty, \sigma\to 0$ provided that $\tau'\ll \tau$. Next, we derive an estimate for $\cJ(u^{\tau,L,\sigma,\tau'})\to 0$ as $\sigma\to 0$ with the other parameters fixed. The conclusion then follows by a diagonal argument. \\ \step{1. Approximation by Relaxed Paths with Finitely Many Values.} We first approximate with a relaxed measure $\Xi^\tau$ which is constant in time. For $\tau>0$, partition the time interval $[0,T]$ by $I_k=[(k-1)\tau,k\tau), 0\le k\le K_\tau=\lceil T/\tau\rceil$, and define \[
\bar\Xi^\tau_{I_k}:=\frac1{|I_k\cap[0,T]|}\int_{I_k\cap[0,T]}\bar\Xi_s\,\dd s,
\]
we obtain a piecewise constant measure
\begin{equation}
\label{eq:pw-constant-relaxed-measure}
\Xi^\tau:=\sum_{k=1}^{K_\tau}\bar\Xi^\tau_{I_k}(\dd x)1(t\in I_k)\,\dd t
\end{equation}
which converges weakly to $\Xi$ as $\tau\downarrow0$. \\ \\ \step{2. Approximation by Regularised Dirac Masses.} Letting $m\ge 0$ be the (constant) mass of $\Xi-c dtdx$, and hence of $\bar\Xi^\tau$, we choose distinct $x^k_\ell$ to attain \begin{equation}\label{eq:approximate-by-point-mass} \begin{split}W\left(\frac{m}{L}\sum_{\ell=1}^{L}\delta_{x^k_\ell},\bar{\Xi}^\tau_{I_k}-cdx\right)& \le \inf_{y^k_\ell \in \TTd} W\left(\frac{m}{L}\sum_{\ell=1}^{L}\delta_{x^k_\ell},\bar{\Xi}^\tau_{I_k}-cdx\right) + L^{-1} \\[1ex]& \le m\delta_L +L^{-1} \end{split}\end{equation} for some sequence $\delta_L\to 0$ as $L\to \infty$. Next, for $\sigma>0$ sufficiently small, consider the smoothed profile
\[
u^{\tau,L,\sigma}_{I_k}:= c+ \frac{m}L\sum_{\ell=1}^L \rho_\sigma(x-x^k_\ell)
\]
where $\rho_\sigma$ is the smoothed point mass \eqref{eq:rho-sigma}, and let $\Xi^{\tau,L,\sigma}$ be the associated relaxed measure, replacing $\bar{\Xi}^\tau_{I_k}$ by $u^{\tau,L,\sigma}_{I_k}$ in \eqref{eq:pw-constant-relaxed-measure}. Provided that $\sigma$ is chosen smaller than $\sigma<\frac12\min_{k,\ell\neq\ell'}|x^k_\ell-x^k_{\ell'}|$, the supports of $\rho_\sigma(\cdot-x^k_l)$ are disjoint for all time, and by construction 
\begin{equation} \label{eq:approximate-by-smooth-mass}
	W\left(u^{\tau,L,\sigma}_{I_k},c+\frac m L \sum_{\ell=1}^L \delta_{x^k_\ell}\right) \le m\sigma.
\end{equation}
\step{3. Approximation by Paths Regular in Time.} Finally, we insert short transition layers.  For every $k=1,\dots, K_\tau-1$, construct a family $\{y^k_\ell(t): 1\le \ell \le L\} \subset C^\infty([0,1],\TTd)$, with all derivatives vanishing at the endpoints, and such that $$y^k_\ell(0)=x^k_\ell\qquad  y^k_\ell(1)=x^{k+1}_\ell $$ and such that $y^k_\ell(t)\neq y^k_{\ell'}(t)$ for any $t\in [0,1], \ell\neq \ell'$. We further restrict $\sigma<\frac12 \min_{k, \ell\neq \ell', t}|y^k_\ell(t)-y^k_{\ell'}(t)|$, and finally define \begin{equation}
	u^{\tau,L,\sigma,\tau'}_t:=\begin{cases} u^{\tau,L,\sigma}_{I_k}	, & t\in I_k \setminus (k\tau-\tau', k\tau]; \\ c+\frac m L \sum_{\ell=1}^L \rho_\sigma(x-y^k_\ell(\frac{t-k\tau+\tau'}{\tau'})), & t\in (k\tau-\tau', k\tau]
\end{cases} \end{equation} and $\Xi^{\tau,L,\sigma,\tau'}(dx,dt):=u^{\tau,L,\sigma,\tau'}_t(x) dx dt$.\\ \\ \step{4. Recovery of the Specified Path $\Xi$.} By construction, each $\Xi^{\tau,L,\sigma,\tau'} \in \cR$ and the density $u^{\tau,L,\sigma,\tau'}\ge c$ everywhere. In this and the subsequent step we verify that, provided the approximation parameters $\tau\to 0, \tau'\to 0, L\to \infty, \sigma\to 0$ in a suitable regime, we have the convergence \begin{equation}
	\label{eq: LSC envelope 1} \Xi^{\tau,L,\sigma,\tau'}\to \Xi
\end{equation} and \begin{equation}
	\label{eq: LSC envelope 2} \cJ(\Xi^{\tau,L,\sigma,\tau'})\to 0.
\end{equation}
First, let $f:[0,T]\times \TTd\to \RR$ be 1-bounded and 1-Lipschitz in both variables, and estimate the error in $\langle f, \Xi\rangle$ at each step of the approximation. Thanks to \eqref{eq:pw-constant-relaxed-measure}, we get $|\langle f, \Xi-\Xi^\tau\rangle| \le (c+m)\tau$, and thanks to (\ref{eq:approximate-by-point-mass}-\ref{eq:approximate-by-smooth-mass}), we obtain $$|\langle f, \Xi^\tau-\Xi^{\tau,L,\sigma}\rangle| \le m(\sigma+\delta_L)+L^{-1}. $$ Finally, the total measure of the times for which $\Xi^{\tau,L,\sigma}$ and $\Xi^{\tau,L,\sigma,\tau'}$ do not coincide is of the order $\tau'(1+(T/\tau))$, and the contribution from these regions is at most $2(m+c)\tau'(1+(T/\tau))$ thanks to the mass bound. Combining everything, and since $f$ was arbitrary, we obtain \begin{equation} \label{eq: final approximation}
	d(\Xi,\Xi^{\tau,L,\sigma,\tau'})\le (c+m)\tau+m(\sigma+\delta_L)+L^{-1}+2(m+c)\tau'\left(1+\frac{T}{\tau}\right) \to 0
\end{equation} provided that the time increments are chosen so that $\tau'\ll \tau$, and \eqref{eq: LSC envelope 1} is proven. \\\\ \step{5. Estimate of Dynamic Cost.} For the second claim \eqref{eq: LSC envelope 2}, we identify a candidate $g$, dividing the times between the intervals where $u^{\tau,L,\sigma,\tau'}$ is stationary, and where the masses are moved. For $t\in I_k\setminus (k\tau-\tau',k\tau]$, we take $$ g:=\frac12 \nabla \log u^{\tau,L,\sigma,\tau'}.$$ Since the supports of the smoothed masses are disjoint, we use Lemma \ref{lem:rho-disipation-est} to estimate, for any $\theta\in (2/d,1)$, \begin{equation}
	\label{eq: competitor g 1} \begin{split}\|g(t)\|^2_{L^2_x}&=\frac14\sum_{l=1}^L \int_{\TTd} \left|\nabla \log \left(c+\frac{m}{L}\rho_\sigma(x-x^k_\ell)\right)\right|^2dx\\ & \le \frac14\sum_{l=1}^L \int_{\TTd} \frac{(m/cL)\rho_\sigma(x-x^k_\ell)}{1+(m/cL)\rho_\sigma(x-x^k_\ell)}\left|\nabla \log \rho_\sigma(x-x^k_\ell)\right|^2dx\\ & \le C L^\theta\sigma^{d\theta-2}. \end{split}
\end{equation} For the transition layers $t\in (k\tau-\tau',k\tau]$, we note that $u^{\tau,L,\sigma,\tau'}$ satisfies \begin{equation}
	\partial_t u^{\tau,L,\sigma,\tau'} =-b(t)\cdot \nabla u^{\tau,L,\sigma,\tau'}
\end{equation} whenever $b_t$ is a vector field equal to \begin{equation}\label{eq: prototype b}(\tau')^{-1}(y^k_\ell)'\left(\frac{t-k\tau+\tau'}{\tau'}\right)\end{equation} on a neighbourhood of the support of $\rho_\sigma(x-y^k_\ell(t))$. Recalling that each $y^k_\ell, y^k_{\ell'}$ are separated by strictly more than $2\sigma$, while the support of $\rho_\sigma(\cdot - y^k_\ell(t))$ is a ball of radius at most $\frac12\sigma$, we may, by standard elliptic regularity theory, construct $b(t)$ satisfying \eqref{eq: prototype b} on the inner balls of radius $\frac{\sigma}{2}$, such that $b(t)$ is supported only on $\cup_\ell \{|x-y^k_\ell(t)|\le\sigma\}$, with $\nabla \cdot b_t=0$ and enjoying the bound \begin{equation}\label{eq: b bound} \|b(t)\|_\infty \le C (\tau')^{-1}\max_{\ell} \left|(y^k_\ell)'\left(\frac{t-k\tau+\tau'}{\tau'}\right)\right|.\end{equation} Consequently, $u^{\tau,L,\sigma,\tau'}$ solves \begin{equation}
	\partial_t u^{\tau,L,\sigma,\tau'} = -\nabla\cdot(b_t u^{\tau,L,\sigma,\tau'})=\frac12\Delta u^{\tau,L,\sigma,\tau'} - \nabla\cdot( u^{\tau,L,\sigma,\tau'} g)
\end{equation}  where we define \begin{equation}
	g:= \frac12\nabla \log u^{\tau,L,\sigma,\tau'} + b_t.
\end{equation} The first term is already estimated in $L^2_x$ by \eqref{eq: competitor g 1}, while using \eqref{eq: b bound} and noting that the size of the support is at most $CL\sigma^d$, the second has norm at most \begin{equation}
	\label{eq: norm of b} \|b(t)\|_{L^2_x}^2 \le C_L (L+1)\sigma^d (\tau')^{-2}
\end{equation} where $C_L$ depends on $y^k_\ell$ and hence $x^k_\ell$ through bounds on the $|(y^k_\ell)'|$ as in  \eqref{eq: b bound}, but is independent of $\tau', \sigma$. Gathering (\ref{eq: competitor g 1}, \ref{eq: norm of b}), we have, for any $\theta\in (2/d,1)$ and some $C(c,m,L,\theta)<\infty$,  \begin{equation}\label{eq: g small}
	\|g\|_{L^2_{t,x}}^2 \le C(c,m,L,\theta)(\sigma^{d\theta-2}+\tau^{-1}(\tau')^{-1}\sigma^d)
\end{equation} which, for fixed $\tau, L, \tau'$ may be made arbitrarily small by sending $\sigma\to 0$. Together with \eqref{eq: final approximation}, the proof is completed by choosing \begin{equation}
	\label{eq: choose parameters} \tau=\frac{T}{n}; \qquad \tau'=\frac{T}{n^2}; \qquad L=n
\end{equation} and choosing $\sigma=\sigma_n \to 0$ fast enough that \eqref{eq: g small} converges to 0 for this choice of $L,\tau'$.\end{proof}

\section*{Acknowledgements} I am grateful to L. Bertini, B. Fehrman, B. Gess and A. Schlichting for interesting and useful discussions of the problem. This work was supported by a Fellowship of the Royal Commission for the Exhibition of 1851.

\bibliographystyle{plain}
\bibliography{literature}

\begin{thebibliography}{10}

\bibitem{adams2013large}
Stefan Adams, Nicolas Dirr, Mark Peletier, and Johannes Zimmer.
\newblock Large deviations and gradient flows.
\newblock {\em Philos. Trans. R. Soc. Lond. Ser. A Math. Phys. Eng. Sci.},
  371(2005):20120341, 17, 2013.

\bibitem{adams2011large}
Stefan Adams, Nicolas Dirr, Mark~A Peletier, and Johannes Zimmer.
\newblock From a large-deviations principle to the wasserstein gradient flow: a
  new micro-macro passage.
\newblock {\em Communications in Mathematical Physics}, 307(3):791--815, 2011.

\bibitem{ambrosio2004transport}
Luigi Ambrosio.
\newblock Transport equation and cauchy problem for bv vector fields.
\newblock {\em Inventiones mathematicae}, 158(2), 2004.

\bibitem{bendahmane2004renormalized}
Mostafa Bendahmane and Kenneth~H Karlsen.
\newblock Renormalized entropy solutions for quasi-linear anisotropic
  degenerate parabolic equations.
\newblock {\em SIAM journal on mathematical analysis}, 36(2):405--422, 2004.

\bibitem{benois1995large}
O~Benois, C~Kipnis, and C~Landim.
\newblock Large deviations from the hydrodynamical limit of mean zero
  asymmetric zero range processes.
\newblock {\em Stochastic processes and their applications}, 55(1):65--89,
  1995.

\bibitem{bernardin2007hydrodynamics}
C{\'e}dric Bernardin.
\newblock Hydrodynamics for a system of harmonic oscillators perturbed by a
  conservative noise.
\newblock {\em Stochastic processes and their applications}, 117(4):487--513,
  2007.

\bibitem{bernardin2008stationary}
C{\'e}dric Bernardin.
\newblock Stationary nonequilibrium properties for a heat conduction model.
\newblock {\em Physical Review E-Statistical, Nonlinear, and Soft Matter
  Physics}, 78(2):021134, 2008.

\bibitem{bernardin2005fourier}
C{\'e}dric Bernardin and Stefano Olla.
\newblock Fourier's law for a microscopic model of heat conduction.
\newblock {\em Journal of Statistical Physics}, 121(3):271--289, 2005.

\bibitem{bertini2006non}
L.~Bertini, A.~De~Sole, D.~Gabrielli, G.~Jona-Lasinio, and C.~Landim.
\newblock Non equilibrium current fluctuations in stochastic lattice gases.
\newblock {\em J. Stat. Phys.}, 123(2):237--276, 2006.

\bibitem{bertini2025private}
Lorenzo Bertini.
\newblock Private communication, 2025.

\bibitem{bertini2015macroscopic}
Lorenzo Bertini, Alberto De~Sole, Davide Gabrielli, Giovanni Jona-Lasinio, and
  Claudio Landim.
\newblock Macroscopic fluctuation theory.
\newblock {\em Rev. Modern Phys.}, 87(2):593--636, 2015.

\bibitem{bertini2005large}
Lorenzo Bertini, Davide Gabrielli, and Joel~L Lebowitz.
\newblock Large deviations for a stochastic model of heat flow.
\newblock {\em Journal of statistical physics}, 121(5):843--885, 2005.

\bibitem{carinci2024solvable}
Gioia Carinci, Chiara Franceschini, Davide Gabrielli, Cristian Giardin{\`a},
  and Dimitrios Tsagkarogiannis.
\newblock Solvable stationary non equilibrium states.
\newblock {\em Journal of Statistical Physics}, 191(1):10, 2024.

\bibitem{carinci2013duality}
Gioia Carinci, Cristian Giardina, Claudio Giberti, and Frank Redig.
\newblock Duality for stochastic models of transport.
\newblock {\em Journal of Statistical Physics}, 152(4):657--697, 2013.

\bibitem{chen2003well}
Gui-Qiang Chen and Beno{\^\i}t Perthame.
\newblock Well-posedness for non-isotropic degenerate parabolic-hyperbolic
  equations.
\newblock In {\em Annales de l'IHP Analyse non lin{\'e}aire}, volume~20, pages
  645--668, 2003.

\bibitem{crippa2008estimates}
Gianluca Crippa, Camillo De~Lellis, et~al.
\newblock Estimates and regularity results for the diperna-lions flow.
\newblock {\em Journal fur die reine und angewandte Mathematik}, 616:15--46,
  2008.

\bibitem{de2008ordinary}
Camillo De~Lellis.
\newblock Ordinary differential equations with rough coefficients and the
  renormalization theorem of ambrosio [after ambrosio, diperna, lions].
\newblock {\em Ast{\'e}risque}, 317(972):175--203, 2008.

\bibitem{de2024hidden}
Anna De~Masi, Pablo~A Ferrari, and Davide Gabrielli.
\newblock Hidden temperature in the kmp model.
\newblock {\em Journal of Statistical Physics}, 191(11):150, 2024.

\bibitem{derrida2007non}
Bernard Derrida.
\newblock Non-equilibrium steady states: fluctuations and large deviations of
  the density and of the current.
\newblock {\em J. Stat. Mech. Theory Exp.}, (7):P07023, 45, 2007.

\bibitem{diperna1989ordinary}
Ronald~J DiPerna and Pierre-Louis Lions.
\newblock Ordinary differential equations, transport theory and sobolev spaces.
\newblock {\em Inventiones mathematicae}, 98(3):511--547, 1989.

\bibitem{dirr2020conservative}
Nicolas Dirr, Benjamin Fehrman, and Benjamin Gess.
\newblock Conservative stochastic pde and fluctuations of the symmetric simple
  exclusion process.
\newblock {\em arXiv preprint arXiv:2012.02126}, 2020.

\bibitem{donsker1989large}
MD~Donsker and SRS Varadhan.
\newblock Large deviations from a hydrodynamic scaling limit.
\newblock {\em Communications on Pure and Applied Mathematics}, 42(3):243--270,
  1989.

\bibitem{duong2013wasserstein}
Manh~Hong Duong, Vaios Laschos, and Michiel Renger.
\newblock Wasserstein gradient flows from large deviations of many-particle
  limits.
\newblock {\em ESAIM: Control, Optimisation and Calculus of Variations},
  19(4):1166--1188, 2013.

\bibitem{fehrman2019large}
Benjamin Fehrman and Benjamin Gess.
\newblock Large deviations for conservative stochastic pde and non-equilibrium
  fluctuations.
\newblock {\em arXiv preprint arXiv:1910.11860}, 2019.

\bibitem{fehrman2021well}
Benjamin Fehrman and Benjamin Gess.
\newblock Well-posedness of the dean-kawasaki and the nonlinear dawson-watanabe
  equation with correlated noise.
\newblock {\em arXiv preprint arXiv:2108.08858}, 2021.

\bibitem{FG22}
Benjamin Fehrman and Benjamin Gess.
\newblock Non-equilibrium large deviations and parabolic-hyperbolic {PDE} with
  irregular drift.
\newblock {\em Invent. Math.}, 234(2):573--636, 2023.

\bibitem{fehrman2025matching}
Benjamin Fehrman, Benjamin Gess, and Daniel Heydecker.
\newblock Matching large deviation bounds of the zero-range process in the
  whole space.
\newblock {\em arXiv preprint arXiv:2507.23452}, 2025.

\bibitem{feng1997microscopic}
Shui Feng, Ian Iscoe, and Timo Sepp{\"a}l{\"a}inen.
\newblock A microscopic mechanism for the porous medium equation.
\newblock {\em Stochastic Processes and their Applications}, 66(2):147--182,
  1997.

\bibitem{gess2023rescaled}
Benjamin Gess and Daniel Heydecker.
\newblock The porous medium equation: Large deviations and gradient flow with
  degenerate and unbounded diffusion.
\newblock {\em arXiv preprint arXiv:2303.11289}, 2023.

\bibitem{gess2026porous}
Benjamin Gess and Daniel Heydecker.
\newblock The porous medium equation: Multiscale integrability in large
  deviations.
\newblock {\em arXiv preprint arXiv:2602.09547}, 2026.

\bibitem{giacomin1999deterministic}
Giambattista Giacomin, Joel~L. Lebowitz, and Errico Presutti.
\newblock Deterministic and stochastic hydrodynamic equations arising from
  simple microscopic model systems.
\newblock In {\em Stochastic partial differential equations: six perspectives},
  volume~64 of {\em Math. Surveys Monogr.}, pages 107--152. Amer. Math. Soc.,
  Providence, RI, 1999.

\bibitem{giardina2009duality}
Cristian Giardina, Jorge Kurchan, Frank Redig, and Kiamars Vafayi.
\newblock Duality and hidden symmetries in interacting particle systems.
\newblock {\em Journal of Statistical Physics}, 135(1):25--55, 2009.

\bibitem{heydecker2023large}
Daniel Heydecker.
\newblock Large deviations of kac's conservative particle system and energy
  nonconserving solutions to the boltzmann equation: A counterexample to the
  predicted rate function.
\newblock {\em The Annals of Applied Probability}, 33(3):1758--1826, 2023.

\bibitem{jabin2010differential}
Pierre-Emmanuel Jabin.
\newblock Differential equations with singular fields.
\newblock {\em Journal de math{\'e}matiques pures et appliqu{\'e}es},
  94(6):597--621, 2010.

\bibitem{jabin2016critical}
Pierre-Emmanuel Jabin.
\newblock Critical non-sobolev regularity for continuity equations with rough
  velocity fields.
\newblock {\em Journal of Differential Equations}, 260(5):4739--4757, 2016.

\bibitem{karlsen2003uniqueness}
Kenneth~Hvistendahl Karlsen and Nils~Henrik Risebro.
\newblock On the uniqueness and stability of entropy solutions of nonlinear
  degenerate parabolic equations with rough coefficients.
\newblock {\em Discrete and Continuous Dynamical Systems}, 9(5):1081--1104,
  2003.

\bibitem{kim2025spectral}
Seonwoo Kim, Matteo Quattropani, and Federico Sau.
\newblock Spectral gap of the kmp and other stochastic exchange models on
  arbitrary graphs.
\newblock {\em arXiv preprint arXiv:2505.02400}, 2025.

\bibitem{kipnis1998scaling}
Claude Kipnis and Claudio Landim.
\newblock {\em Scaling limits of interacting particle systems}, volume 320.
\newblock Springer Science \& Business Media, 1998.

\bibitem{kipnis1982heat}
Claude Kipnis, Carlo Marchioro, and Errico Presutti.
\newblock Heat flow in an exactly solvable model.
\newblock {\em Journal of Statistical Physics}, 27(1):65--74, 1982.

\bibitem{kipnis1989hydrodynamics}
Claude Kipnis, Stefano Olla, and SRS Varadhan.
\newblock Hydrodynamics and large deviation for simple exclusion processes.
\newblock {\em Communications on Pure and Applied Mathematics}, 42(2):115--137,
  1989.

\bibitem{ladyzhenskaia1968linear}
Olga~Aleksandrovna Ladyzhenskaia, Vsevolod~Alekseevich Solonnikov, and Nina~N
  Ural'tseva.
\newblock {\em Linear and quasi-linear equations of parabolic type}, volume~23.
\newblock American Mathematical Soc., 1968.

\bibitem{lellis2003structure}
Camillo~De Lellis, Felix Otto, and Michael Westdickenberg.
\newblock Structure of entropy solutions for multi-dimensional scalar
  conservation laws.
\newblock {\em Archive for rational mechanics and analysis}, 170(2):137--184,
  2003.

\bibitem{LPT94}
P.-L. Lions, B.~Perthame, and E.~Tadmor.
\newblock A kinetic formulation of multidimensional scalar conservation laws
  and related equations.
\newblock {\em J. Amer. Math. Soc.}, 7(1):169--191, 1994.

\bibitem{mariani2010large}
Mauro Mariani.
\newblock Large deviations principles for stochastic scalar conservation laws.
\newblock {\em Probability theory and related fields}, 147:607--648, 2010.

\bibitem{modena2020convex}
Stefano Modena and Gabriel Sattig.
\newblock Convex integration solutions to the transport equation with full
  dimensional concentration.
\newblock In {\em Annales de l'Institut Henri Poincar{\'e} C, Analyse non
  lin{\'e}aire}, volume~37, pages 1075--1108. Elsevier, 2020.

\bibitem{modena2018non}
Stefano Modena and L{\'a}szl{\'o} Sz{\'e}kelyhidi~Jr.
\newblock Non-uniqueness for the transport equation with sobolev vector fields.
\newblock {\em Annals of PDE}, 4(2):18, 2018.

\bibitem{peletier2014large}
Mark~A Peletier, Frank Redig, and Kiamars Vafayi.
\newblock Large deviations in stochastic heat-conduction processes provide a
  gradient-flow structure for heat conduction.
\newblock {\em Journal of Mathematical Physics}, 55(9):093301, 2014.

\bibitem{P02}
Beno\^{\i}t Perthame.
\newblock {\em Kinetic formulation of conservation laws}, volume~21 of {\em
  Oxford Lecture Series in Mathematics and its Applications}.
\newblock Oxford University Press, Oxford, 2002.

\bibitem{quastel1999large}
Jeremy Quastel, Fraydoun Rezakhanlou, and SR~Srinivasa Varadhan.
\newblock Large deviations for the symmetric simple exclusion process in
  dimensions d$\ge$ 3.
\newblock {\em Probability theory and related fields}, 113(1):1--84, 1999.

\bibitem{QY}
Jeremy Quastel and Horng-Tzer Yau.
\newblock Lattice gases, large deviations, and the incompressible
  {N}avier-{S}tokes equations.
\newblock {\em Ann. of Math. (2)}, 148(1):51--108, 1998.

\bibitem{suzuki1993hydrodynamic}
Yuki Suzuki and K\^{o}hei Uchiyama.
\newblock Hydrodynamic limit for a spin system on a multidimensional lattice.
\newblock {\em Probab. Theory Related Fields}, 95(1):47--74, 1993.

\bibitem{varadhan1997diffusive}
SRS Varadhan and Horng-Tzer Yau.
\newblock Diffusive limit of lattice gas with mixing conditions.
\newblock {\em Asian Journal of Mathematics}.

\end{thebibliography}

\end{document}